\documentclass[10pt]{article}      

\usepackage{amsmath,amssymb, amsfonts,amsthm,amscd} 
\usepackage{bbold, dsfont, mathrsfs}
\usepackage{graphicx}
\usepackage{xspace}
\usepackage{coffeestains}
\usepackage{nccmath}   
\usepackage[hidelinks]{hyperref}
\usepackage{enumerate}  
\usepackage{relsize} 
\usepackage{xparse}  
\usepackage{multirow} 

\usepackage[normalem]{ulem}
\makeatletter \renewcommand{\p@enumii}{} \makeatother

\usepackage{latexsym}
\usepackage{tikz}
\usetikzlibrary{tikzmark}

\theoremstyle{plain}
\newtheorem{theorem}{Theorem}[section]
\newtheorem{lemma}[theorem]{Lemma}
\newtheorem{proposition}[theorem]{Proposition}
\newtheorem{corollary}[theorem]{Corollary}
\newtheorem{assumption}[theorem]{Assumption}
\newtheorem{remark}[theorem]{Remark}
\theoremstyle{remark}
\AtEndEnvironment{remark}{\qed}

\usepackage{xcolor}

\usepackage[style=alphabetic, giveninits=true]{biblatex}
\def\A{\mathcal{A}}

\def\EE{\mathbb E}

\def\M{\mathcal{M}}
\def\NN{\mathbb{N}}
\def\bigoh{\mathcal{O}}
\def\P{\mathcal{P}}
\def\PP{\mathbb{P}}
\def\Q{\mathcal{Q}}

\def\RR{\mathbb{R}}

\def\TT{\mathbb{T}}

\def\ZZ{\mathbb{Z}}

\newcommand{\Rho}{\mathlarger{\mathlarger{\rho}}}

\DeclareMathOperator{\Dom}{Dom}

\DeclareMathOperator*{\sumx}{\ensuremath{\sum_{x\in\TT_\epsilon^d}}}

\DeclareMathOperator*{\sumsum}{\sum\!\sum}
\DeclareMathOperator*{\sumsumsum}{\sum\!\sum\!\sum}

\newcommand{\simsum}[1][]{\sumsum_{\substack{x_{#1},y_{#1}\in\TT_\epsilon^d\\ x_{#1}\sim y_{#1}}}}

\DeclareMathOperator{\supp}{supp}
\def\TV{\mathrm{TV}}

\def\llangle{\langle\!\langle}
\def\rrangle{\rangle\!\rangle}

\newcommand{\lle}{\ll\hspace{-1.2em}\raisebox{-0.45em}{$\sim$}\;}
\newcommand{\gge}{\gg\hspace{-1.2em}\raisebox{-0.45em}{$\sim$}\;}

\newcommand{\nnp}[1]{\stackrel{#1}{\leadsto}}

\newcommand{\super}[1]{^{\scriptscriptstyle{(#1)}}}

\def\div{\mathop{\mathrm{div}}\nolimits}
\newcommand{\grad}{\nabla}

\NewDocumentCommand{\etaV}{O{x} O{t} O{} O{}}{%
  \ensuremath{\eta_{#4}^{#3}(#1,#2)}%
}

\NewDocumentCommand{\rhoV}{O{x} O{t} O{} O{}}{%
  \ensuremath{\rho_{#4}^{#3}(#1,#2)}%
}

\NewDocumentCommand{\empmeas}{O{dx} O{t} O{} O{}}{%
  \ensuremath{\Rho_{#4}^{#3}(#1,#2)}%
}

\NewDocumentCommand{\corr}{O{x} O{t} O{}}{%
  \ensuremath{\rho^{\langle#3\rangle}(#1,#2)}%
}
\NewDocumentCommand{\corrv}{O{x} O{t} O{}}{%
  \ensuremath{\nu^{\langle#3\rangle}(#1,#2)}%
}

\NewDocumentCommand{\corrmic}{O{x} O{t} O{}}{%
  \ensuremath{\rho^{(n),\langle#3\rangle}(#1,#2)}%
}

\newcommand{\IRW}{\ensuremath{\mathrm{IRW}}\xspace}
\newcommand{\SEP}{\ensuremath{\mathrm{SEP}}\xspace}

\newcommand{\EP}{\ensuremath{\mathrm{SEP}}\xspace}
\newcommand{\STIR}{\ensuremath{\mathrm{stir}}\xspace}
\newcommand{\BM}{\ensuremath{\mathrm{bm}}\xspace}

\DeclareMathAlphabet{\mathmybb}{U}{bbold}{m}{n}

\makeatletter
  \newcommand{\eqnum}{\leavevmode\hfill\refstepcounter{equation}\textup{\tagform@{\theequation}}} 
\makeatother

\title{Hydrodynamic limits of collisions and fluxes in the exclusion process}
\author{%
Mario Ayala \thanks{TU M\"unchen, Boltzmannstra{\ss}e 3, 85747 Garching, Germany. Email: \href{mailto:d.r.m.renger@tum.de}{mario.ayala@tum.de}}
\and
D.R.\ Michiel Renger\thanks{TU M\"unchen, Boltzmannstra{\ss}e 3, 85747 Garching, Germany. Email: \href{mailto:d.r.m.renger@tum.de}{d.r.m.renger@tum.de}}
}

\begin{document}
\maketitle
\begin{abstract}
We extend the usual hydrodynamic description of the symmetric exclusion process by keeping track of collision events corresponding to jumps into already occupied sites, thereby quantifying the dissipated part of the microscopic activity that is otherwise discarded by the empirical density in the macroscopic limit. In addition to the empirical density and net current, we study unidirectional fluxes and collision counts under flexible joint scalings of the lattice spacing and particle number. These collision and flux observables have regime-dependent hydrodynamic limits, with deterministic unidirectional behaviour and a stochastic space–time white noise limit for the net collision count. Our results provide a quantitative decomposition of exclusion dynamics into transport and collision effects and clarify how microscopic blocking manifests at the macroscopic and fluctuation levels.
\end{abstract}

\section{Introduction}
\label{sec:intro}

Two commonly used particle systems to model diffusion phenomena are the symmetric exclusion process (\EP), introduced by Spitzer in \cite{spitzer1970interaction}, and the system of independent random walkers (\IRW). Unlike Hamiltonian mechanics that rely on Newton’s laws, both particle dynamics are described by simple stochastic rules at the microscopic scale, making it easy to pass to the hydrodynamic limit at the macroscopic scale.

Remarkably, the empirical measure of both \EP and \IRW have the exact same hydrodynamic limit $\partial_t\rho=\Delta\rho$ \cite{KipnisOllaVaradhan1989,KipnisLandim1999}. Classically, the difference between the two models can be seen by introducing a `weak asymmetry' in the system, yielding additional terms $-\div(\rho A)$ or $-\div(\rho\,(1-\rho)A)$ in the hydrodynamic limit for \IRW and \EP respectively \cite{yau1991relative}. However, this approach observes the empirical measure only, and so does not capture all the nuances of the microscopic dynamics.

In this work we aim to provide a deeper understanding of differences between the two models, by studying other observables in their hydrodynamic limit. These observables are motivated by the following differences in the microscopic dynamics.

Both \EP and \IRW consist of $n$ particles jumping on a lattice. In the \IRW model, one particle and jump direction are randomly uniformly selected, causing that particle to jump to the nearest neighbour site in that direction. The \EP differs from \IRW only in that the particle will not jump if the target site is already occupied by another particle. The differences between the two models can be seen from two different perspectives. 
\begin{enumerate}[A.]
\item From a thermodynamic perspective, randomly picking a particle and direction means extracting a little free energy from a heat bath. In the case of \IRW this free energy drives the macroscopic dynamics, which is consistent with the Wasserstein picture of diffusion~\cite{Jordan1998}. In the case of the \EP however, selected jumps to occupied sites do not occur, and so that extracted energy is lost. We interpret this loss of energy as dissipation due to an inelastic \emph{collision} between the selected particle and the particle occupying the target site. Hence the dissipated energy can be measured either by the \emph{collision number} $\widetilde C_{xy}(t)$, counting the number of collisions between any two sites $x,y$, or by the \emph{unidirectional flux} $\widetilde W_{xy}(t)$, counting the number of jumps that actually did occur. These observables are directly related to unidirectional flux $\widetilde{W}^\IRW$ of the \IRW model: a particle that would have jumped in the \IRW model either jumps in the \EP model or causes a collision, and so, locally in time,
\begin{align}
  \partial_t{\widetilde W}^\IRW_{xy}(t)=\partial_t{\widetilde W}_{xy}(t) + \partial_t{\widetilde C}_{xy}(t),
\label{eq:unidirectional difference}
\end{align}

\item From a dynamic perspective, one commonly studies the \emph{net flux} $\bar{\widetilde{W}}_{xy}(t):=\widetilde{W}_{xy}(t)-\widetilde{W}_{yx}(t)$ of the \EP, since this determines the evolution of the empirical measure through the continuity equation, and can be used to measure non-equilibrium behaviour~\cite{BDSGJLL2015MFT}. Then one can similarly introduce net collisions in \EP and net fluxes in \IRW to obtain the net analogue of \eqref{eq:unidirectional difference}:
\begin{align}
  \partial_t{\bar{\widetilde W}}^\IRW_{xy}(t)=\partial_t{\bar{\widetilde W}}_{xy}(t) + \partial_t{\bar{\widetilde C}}_{xy}(t).
\label{eq:net difference}
\end{align}
\end{enumerate}
The aim of this paper is to derive the hydrodynamic limits of \eqref{eq:unidirectional difference} and \eqref{eq:net difference}. To obtain this, we need to show the hydrodynamic limit of the \emph{five} quantities $(\eta,\widetilde W,\widetilde C,\bar{\widetilde W}, \bar{\widetilde C})$. The hydrodynamic limit for the particle configuration $\eta$ and net flux $\bar{\widetilde W}$ are classic~\cite{KipnisOllaVaradhan1989,KipnisLandim1999, BDSGJLL2006,BDSGJLL2007}; the hydrodynamic limits of the unidirectional flux as well as the unidirectional and net collisions appear to be new.

Naturally, hydrodynamic limits are only meaningful after appropriate rescaling. With regard to rescaling, we allow for a certain degree of flexibility that does not seem to be common in the literature on exclusion processes. More specifically, we allow for \emph{two} model parameters. On the one hand, $\epsilon>0$ will denote the lattice spacing, and particle jump rates will be scaled parabolically with $1/\epsilon^2$. On the other hand, $n$ will control the number of particles in the system; this number will determine the rescaling in the empirical measure. The key relation between the two parameters is given by the total number of particles per site $\epsilon^d n$ (we work on the discrete torus with $\epsilon^{-d}$ sites). The classic regime from the literature is $n \sim 1/\epsilon^d$. The ``clogged'' regime $n\gg 1/\epsilon^d$ is not allowed because there are not enough sites to fill. In the ``sparse'' regime $n \ll 1/\epsilon^d$ the collisions become rarer and so the system behaves more like the independent $\IRW$ model. It turns out that this choice of rescaling can have a subtle impact on the hydrodynamic limits and so we have to make a more detailed distinction between scaling regimes.

\paragraph{Outline.}

In Section~\ref{sec:model} we introduce the model and variables in more detail, as well as the scaling regimes and notation.
In Section~\ref{sec: metatheorems} we collect deterministic and stochastic hydrodynamic limit criteria for distribution-valued Markov processes, which are used as a black box throughout the paper. We present these results in a unified and self-contained form, with precise assumptions tailored to the Skorohod--Mitoma topology. Although these results are largely standard, they are often only implicitly used in the literature; we include them here to make the framework explicit and precise.
Section~\ref{sec:correlation functions} contains the core technical estimates: convergence and uniform bounds for $k$-points and time-integrated correlation functions of the SEP, including their dependence on the scaling regime.
In Section~\ref{sec:tightness} we prove tightness (and exponential tightness where applicable) of all observables of interest, taking the precise scaling regimes into account.
Section~\ref{sec:mainresults} contains the main hydrodynamic limit results of the paper. 
Using the abstract criteria from Section~\ref{sec: metatheorems}, the proofs reduce to verifying the corresponding tightness, drift, and quadratic variation conditions for the empirical measure, fluxes, and collision observables, leading in particular to a stochastic limit for the net collision number. 
Section~\ref{sec: disc} concludes with a discussion of the results and possible extensions. The appendices collect background material on the functional-analytic and probabilistic setting used throughout the paper, including the distributional framework and martingale tools required for the proofs.

\section{Model and notation}
\label{sec:model}

We introduce the precise particle system and variables, starting with basic notation in Subsection~\ref{subsec:notation}. 
In Subsection~\ref{subsec:dynamics} we define the microscopic dynamics and their macroscopic observables of interest, including the empirical measure and the various flux and collision variables. These are defined as measure-valued processes embedded into a common distributional space to facilitate convergence results. Subsection~\ref{subsec:initial distributions} introduces the class of initial conditions: the slowly-varying product distributions which allow us to derive the correct hydrodynamic scalings. Subsection~\ref{subsec:scaling} discusses the admissible scaling regimes between the lattice spacing $\epsilon$ and the number of particles $n$, and their implications for the limiting behaviour of the system.

\subsection{Notation and spaces}
\label{subsec:notation}

We use upper- and lowercases to distinguish between random variables and deterministic objects. In particular, the empirical measure will be denoted by the ``upper case'' $\Rho\super{n}$.
Throughout the paper $d\in\NN$ is a fixed dimension, $T>0$ a fixed end time. \\

For time-dependent measures we write $\rho(t)$ and $\rho(dx,t)$ interchangeably, and for absolutely continuous measures as usual $\rho(dx,t)=\rho(x,t)\,dx$.

To keep the presentation tidy and because our main focus is not on `mass escaping to infinity', we work on the compact, continuous torus $\TT^d$. The corresponding discrete torus is denoted by $\TT_\epsilon^d:=(\epsilon\NN/\ZZ)^d$. On the discrete torus, we use the following notations for nearest neighbours:
\begin{align*}
    x\sim y &&:\! \iff & y=x\pm\epsilon\mathds1_l \qquad\text{ for some } l=1,\hdots d, \pm=+,-,\\
    x\nnp{} y &&:\!\iff& y=x+\epsilon\mathds1_l \qquad\text{ for some } l=1,\hdots d,\\
    x\nnp{l} y &&:\!\iff& y=x+\epsilon\mathds1_l.
\end{align*}
With this notation we distinguish between nearest neighbours $x\sim y$ in all directions, and nearest neighbours in positive directions $x\leadsto y$, see Figure~\ref{fig:nearest neighbours}. In particular, sums over nearest-neighbouring pairs $\sum_{x,y\in\TT_\epsilon^d, x\sim y}$ include both pairs $(x,y)$ and $(y,x)$, whereas in sums over $x\nnp{} y$ each pair $(x,y)$ is only counted once.
\begin{figure}[h!]
\begin{minipage}{0.5\textwidth}
\begin{center}
\begin{tikzpicture}[scale=0.8]
  \tikzstyle{every node}=[font=\scriptsize]
  \draw[gray](-1.3,-1)--(1.3,-1);
  \draw[gray](-1.3,0)--(1.3,0);
  \draw[gray](-1.3,1)--(1.3,1);
  \draw[gray](-1,-1.3)--(-1,1.3);
  \draw[gray](0,-1.3)--(0,1.3);
  \draw[gray](1,-1.3)--(1,1.3);
  \filldraw (0,0) circle(0.15) node[below=0.15cm]{$x$};
  \filldraw (0,-1) circle(0.15) node[right=0.15cm]{$y$};
  \filldraw (-1,0) circle(0.15) node[below=0.15cm]{$y$};
  \filldraw (1,0) circle(0.15) node[below=0.15cm]{$y$};
  \filldraw (0,1) circle(0.15) node[right=0.15cm]{$y$};    
\end{tikzpicture}
\end{center}
\end{minipage}
\begin{minipage}{0.5\textwidth}
\begin{center}
\begin{tikzpicture}[scale=0.8]
  \tikzstyle{every node}=[font=\scriptsize]
  \draw[gray](-1.3,-1)--(1.3,-1);
  \draw[gray](-1.3,0)--(1.3,0);
  \draw[gray](-1.3,1)--(1.3,1);
  \draw[gray](-1,-1.3)--(-1,1.3);
  \draw[gray](0,-1.3)--(0,1.3);
  \draw[gray](1,-1.3)--(1,1.3);
  \filldraw (0,0) circle(0.15) node[below=0.15cm]{$x$};
  \filldraw (1,0) circle(0.15) node[below=0.15cm]{$y$};  
  \filldraw (0,1) circle(0.15) node[right=0.15cm]{$y$};      
\end{tikzpicture}
\end{center}
\end{minipage}
\caption{In dimension $d=2$: all nearest neighbours $x\sim y$ of $x$ (left), and all nearest neighbours in positive direction $x\leadsto y$ (right).}
\label{fig:nearest neighbours}
\end{figure}
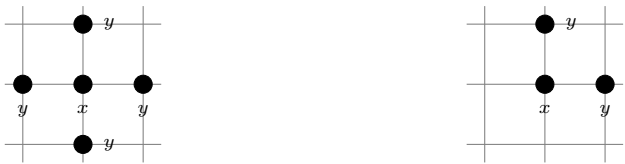

For consistency we use the same topology for all variables, by embedding the space of measures $\M(\TT^d)$ into the space of distributions $C(\TT^d)^*$, see Subsection~\ref{subsec:distro}. Narrow and vague convergence of a sequence of measures in $\M(\TT^d)$ to a measure are equivalent, which are both also equivalent to weak* and strong convergence of distributions. 

Lastly, we introduce the following duality pairings:
\begin{subequations}
\begin{align}
  \langle \phi,\rho\rangle_{x}&:=\int_{\TT^d}\!\phi(x)\,\rho(dx),
  &&\phi\in C(\TT^d),\\
  \langle \phi,\bar{w}\rangle_{x,l}&:=\sum_{l=1}^d \int_{\TT^d}\!\phi_l(x)\,\bar{w}_l(dx),
  &&\phi\in C(\TT^d\times\{1,\hdots,d\}),\\  
  \langle \phi,w\rangle_{x,l,\pm}&:=\sum_{l=1}^d \sum_{\pm=+,-} \int_{\TT^d}\!\phi_{l,\pm}(x)\,w_{l,\pm}(dx),
  &&\phi\in C(\TT^d\times\{1,\hdots,d\}\times\{+,-\}).
\end{align}
\label{eq:brackets notations}
\end{subequations}

\subsection{Dynamics}
\label{subsec:dynamics}

The microscopic variables are: the particle configuration $\eta\in\{0,1\}^{\TT_\epsilon^d}$, the cumulative particle flux $\widetilde{W}\in(\NN_0)^{\TT_\epsilon^d\times\TT_\epsilon^d}$ and the collision number $\widetilde C\in(\NN_0)^{\TT_\epsilon^d\times\TT_\epsilon^d}$. The triple $(\eta(t),\widetilde{W}(t),\widetilde{C}(t))$ is then a Markov process with generator:
\begin{multline}
  (\widetilde{\Q}\super{n} \Phi)(\eta,\widetilde w,\widetilde c)\\
  =\frac1{\epsilon^2}\simsum \eta(x)\,(1-\eta(y))\big\lbrack \Phi(\eta-\mathds1_x+\mathds1_y,\widetilde{w}+\mathds1_{xy},\widetilde{c})-\Phi(\eta,\widetilde w,\widetilde c)\big\rbrack\\
  +\frac1{\epsilon^2}\simsum \eta(x)\,\eta(y)\big\lbrack \Phi(\eta,\widetilde{w},\widetilde{c}+\mathds1_{xy})-\Phi(\eta,\widetilde w,\widetilde c)\big\rbrack.
\label{eq:raw generator}
\end{multline}
The interpretation is the following: after picking a position with a particle ($\eta(x)=1$) and a target site $y$, either the particle jumps if the target site is empty ($\eta(y)=0$), causing an update in the configuration and flux, or, if the target site is occupied ($\eta(y)=1$) the collision number is updated. Observe that the variables are already rescaled in space with $\epsilon$ and in time with $1/\epsilon^2$.

In order to pass to the hydrodynamic limit we rescale with $n$ as well, where the superscript $n$ always signifies dependency on $n$ as well as $\epsilon$, assuming $\epsilon=\epsilon_n\to0$ as $n\to\infty$. This leads to the following five macroscopic variables: \emph{empirical measure, unidirectional flux, unidirectional collision number, net flux and net collision number}.
\begin{align*}
    \Rho\super{n}(t)&:=\frac1n\sum_{x\in\TT_\epsilon^d} \etaV\delta_x,\\
    W\super{n}_{l,\pm}(t)&:=\frac{\epsilon^2}{n}\sum_{x\in\TT_\epsilon^d}\widetilde W_{x,x\pm\epsilon\mathds1_l}(t)\delta_{x\pm\epsilon/2\mathds1_l}, 
      &l=1,\hdots d,\, \pm= +,-,\\
    C\super{n}_{l,\pm}(t)&:=\frac{\epsilon^2}{\epsilon^d n^2} \sum_{x\in\TT_\epsilon^d} \widetilde C_{x,x\pm\epsilon\mathds1_l}(t)\delta_{x\pm\epsilon/2\mathds1_l}, 
      &l=1,\hdots d,\, \pm=+,-,\\
    \bar{W}\super{n}_l(t)&:=\frac{\epsilon}{n}\sum_{x\in\TT_\epsilon^d}\big(\widetilde W_{x,x+\epsilon\mathds1_l}(t)-\widetilde W_{x+\epsilon\mathds1_l,x}(t)\big)\delta_{x+\epsilon/2\mathds1_l}, 
      &l=1,\hdots d,\\ 
    \bar{C}\super{n}_l(t)&:=\frac{\epsilon}{\sqrt{\epsilon^d}n}\sum_{x\in\TT_\epsilon^d}\big(\widetilde C_{x,x+\epsilon\mathds1_l}(t)-\widetilde C_{x+\epsilon\mathds1_l,x}(t)\big)\delta_{x+\epsilon/2\mathds1_l}, 
      &l=1,\hdots d.
\end{align*}
A few remarks are in order here. First, observe that the net fluxes and collisions are defined on $d$ directions, and the unidirectional ones on $2d$ directions. Second, we always place collisions and fluxes at the center point $x+\epsilon/2\mathds1_l$ of the edge. Third, the particular prefactors are of course chosen precisely as to obtain nontrivial limits. The square root in the prefactor of $\bar{C}\super{n}(t)$ may be interpreted as a central-limit scaling; indeed this variable will turn out to have a stochastic limit. However, under different scaling regimes the other variables can also have stochastic limits. Finally, due to the different scalings, the limit of $\bar W\super{n}(t)$ is not a function of the limit of $W\super{n}(t)$ and similarly for $\bar C\super{n}(t),C\super{n}(t)$.

\subsection{Initial and path measure}\label{subsec:initial distributions}

In order to define the path measure we first introduce the initial condition that we use.

\begin{assumption}[Slowly varying parameter/ product distribution] Let $\rho_0\in C(\TT^d)$ be a given (deterministic) continuous density. The initial configurations $(\etaV[x][0])$ at each lattice site $x \in\TT_\epsilon^d$ are independently Bernoulli distributed with parameter
\begin{align*}
    n\int_{B_\epsilon(x)}\!\rho_0(y)\,dy\leq1, &&
    B_\epsilon(x):=\lbrack x_1-\tfrac{\epsilon}{2},x_1+\tfrac{\epsilon}{2}\rbrack \times \hdots \times \lbrack x_d-\tfrac{\epsilon}{2},x_d+\tfrac{\epsilon}{2}\rbrack.
\end{align*}
Furthermore, the cumulative variables are initially (almost surely):
\begin{align}
    W\super{n}(0)&:=0,
    &&
    C\super{n}(0):=0,
    &&
    \bar{W}\super{n}(0):=0,
    &&
    \bar{C}\super{n}(0):=0.
\end{align}
\label{ass:initial}
\end{assumption}

\begin{remark}\label{Rem: Ass initial corr} Our results in the standard scaling regime remain valid if we generalise the initial configuration to, for distinct points $x^i\in\TT^d$:
\begin{align*}
    \Rho\super{n}(0)\xrightarrow[\text{narrow}]{n\to\infty}\rho_0, &&\text{and}&& \frac1{(\epsilon^d n)^k} \, \EE \left[ \prod_{i=1}^k \etaV[x_\epsilon^i][0] \right] \xrightarrow{n\to\infty} \prod_{i=1}^k \rho_0(x^i),
    &&k=1,2,3,4.
\end{align*} 
Since this is only a slight generalisation we work with the slowly-varying-parameter assumption instead.\\
However, the hydrodynamic limits will really fail as soon as either one of the four correlation function is not a product measure in the limit.
\end{remark}

We denote by $\PP$ and $\EE$ the path probability measure and expectation of the microscopic variables $\eta(t)$, $\widetilde{W}(t)$, $\widetilde{C}(t)$, which also induces the path probability measure and expectation of the macroscopic variables $\Rho\super{n}(t)$, $W\super{n}(t)$, $C\super{n}(t)$, $\bar{W}\super{n}(t)$, $\bar{C}\super{n}(t)$. Although for finite $n$ these variables are all measure-valued, the limit of $\bar{C}\super{n}$ will be a random distribution, and so for consistency we embed all variables into the space of distributions, see Subsection~\ref{subsec:distro},
\begin{align*}
    \mathrm{Dist} := &
    C^\infty(\TT^d)^* 
      \times 
    C^\infty(\TT^d \times \lbrace 1, \ldots, d\rbrace \times \lbrace +, -\rbrace)^*\\
      &\qquad\times
    C^\infty(\TT^d \times \lbrace 1, \ldots, d\rbrace \times \lbrace +, -\rbrace)^*
      \times
    C^\infty(\TT^d \times \lbrace 1, \ldots, d\rbrace)^*\\
      &\qquad\times
    C^\infty(\TT^d \times \lbrace 1, \ldots, d\rbrace)^*.
\end{align*}
As such, $\PP$ will denote a Borel probability measure on the Skorohod-Mitoma space $D(0,T;\mathrm{Dist})$ of c{\`a}dl{\`a}g paths, see Subsection~\ref{subsec:Skorohod-Mitoma}.

\subsection{Scaling regimes}
\label{subsec:scaling}

Here we discuss the relation between the two model parameters in more detail, namely: the lattice spacing $\epsilon$, and the order of the number of particles $n$, corresponding to the initial configuration $\eta(0)$. We always asssume that
\begin{align*}
    \epsilon=\epsilon_n\to0 \qquad\text{ as } n\to\infty,
\end{align*}
but allow for a certain amount of flexibility in the precise scaling $\epsilon_n$. However, this flexibility is restricted by the following observations.

On the one hand, the initial empirical measure $\Rho\super{n}(0)$ converges narrowly to the fixed, continuous density $\rho_0\in C(\TT^d)$ from Assumption~\ref{ass:initial}. This implies that the total number of particles is approximately $n\lVert\rho_0\rVert_{L^1(\TT^d)}$. On the other hand, since we have Bernouilli random variables at each site, the total number of particles cannot exceed the number $1/\epsilon^d$ of lattice sites. These observations lead to different possible scaling regimes ($\alpha$ signifies the expected fraction of particles per lattice site):
\begin{itemize}
\item classic regime: $n\sim \frac1{\epsilon^d}$ $:\iff$ $\epsilon^d n \to\alpha/\lVert\rho_0\rVert_{L^1(\TT^d)}$ for some $\alpha\in(0,1)$,
\item sparse regime: $n\ll\frac{1}{\epsilon^d}$ $:\iff$ $\epsilon^d n\to 0$ \qquad  ($=\alpha/\lVert\rho_0\rVert_{L^1(\TT^d)}$ for $\alpha=0$).
\end{itemize}
Feasible scalings are either one of these regimes, and will be denoted as $n\lle 1/\epsilon^d$. 

The classic regime is commonly studied in the literature, the sparse regime much less. However, in this paper we argue that different, non-classic regimes are very interesting because they can really influence tightness properties as well as the hydrodynamic limits. This seems to be a new perspective in the context of exclusion processes (see for example \cite{PfaffelhuberPopovic2015} in the context of reaction-diffusion compartment models, or \cite{PulvirentiSimonella2017} in the context of kinetic theory). 
It turns out that we need to impose additional restrictions to the scaling regime. One such restriction is the following. In order for the collision numbers $C\super{n}$ and $\bar{C}\super{n}$ to have functional laws of large numbers, the prefactors $\epsilon^2/(\epsilon^d n^2)$ and $\epsilon/(\sqrt{\epsilon^d}n)$ would need to vanish, which requires $\epsilon/\sqrt{\epsilon^d}\ll n\lle 1/\epsilon^{d}$. However, there is another issue that restricts us even more, to the regimes:
\begin{align*}
  \frac{1}{\sqrt{\epsilon^d}} \ll n\lle \frac{1}{\epsilon^{d}}.
\end{align*}
This issue is due to a fundamental problem with the initial condition: without this restriction, the initial condition would allow too many ``occupied nearest neighbours'', which destroys our tightness as well as our limit theorems (see for example Corollary~\ref{cor:kptsconvintegral}).

Working with the non-classical regimes is mostly straight-forward, with one difference to bare in mind. In the literature it is often used that the limiting density $\rho$ of the empirical measure $\Rho\super{n}$ is a priori bounded:
\begin{proposition} Assume $n\sim 1/\epsilon^d$. Then $\rho(x,t)\leq\frac{\lVert\rho_0\rVert_{L^1(\TT^d)}}{\alpha}$ at all $x\in\TT^d, t\geq0$.
\end{proposition}
\begin{proof}
Let $R_\delta(x)$ be a small open rectangle with width $\delta>0$ and center $x\in\TT^d$. For arbitrary $x\in\RR^d,\delta,\tau>0$, take an arbitrary function $\phi\in C(\TT^d)$ so that $\mathds1_{R_\delta(x)}\leq\phi\leq \mathds1_{R_{\delta+\tau}(x)}$. We first interpret $\rho(dx)=\rho(x,t)\,dx$ as a measure, and ignore the time-dependence for brevity. The narrow convergence $\Rho\super{n}\rightharpoonup\rho$ and the fact that each site cannot be occupied by more than one particle imply that:
\begin{align*}
  \rho(R_\delta(x)) &\leq\langle\phi,\rho\rangle 
  = \lim_{n\to\infty} \EE\langle\phi,\Rho\super{n}\rangle
   \leq \lim_{n\to\infty} \EE\Rho\super{n}\big(eR_{\delta+\tau}(x)\big)\\
   &\leq \lim_{n\to\infty} \#\big((\TT_\epsilon^d\cap R_{\delta+\tau}(x)\big)
   \leq \lim_{n\to\infty} \frac{\lvert R_{\delta+\tau}(x)\rvert}{\epsilon^d n}
   = \frac{\lvert R_{\delta+\tau}(x)\rvert \,\rVert\rho_0\rVert_{L^1(\TT^d)}}{\alpha}
\end{align*}
Since $\tau>0$ is arbitrary we have the same inequality with $\tau=0$. Dividing by $\lvert R_\delta(x)\rvert$ on both sides and letting $\delta\to0$ produces the claim for the density.
\end{proof}
In this argument the role of $\alpha\in(0,1)$ is explicit, and so we see why the a priori bound fails for sparse regimes $\alpha=0$. A posteriori, the limiting density $\rho$ is the solution to the diffusion equation on the torus and hence bounded, but it is unclear whether for example large deviations can allow unbounded densities or even singular measures.


\section{Metatheorems}\label{sec: metatheorems}

This section collects essential arguments that we need to pass to deterministic and stochastic hydrodynamic limits (Subsecs.~\ref{subsec:meta det hydrodyn} and \ref{subsec:meta stoch hydrodyn}), and to show tightness (Subsec.~\ref{subsec:meta tightness}). All of these results are standard, but we reformulated and included them here since we felt that a snappy but rigorous overview was lacking in the literature. We hope that this helps readers understand our main arguments without getting lost in details, as well as provide a starting point for showing different hydrodynamic limits in the future.

Throughout \underline{this section} we shall work in a more general setting, but to keep notation tidy we restrict to \emph{distribution-valued} Markov processes $(Z\super{n}(t))_{t\in\lbrack0,T\rbrack}$ with c{\`a}dl{\`a}g paths $D(0,T;C^\infty(\TT^d)^*)$, the latter being equipped with the Skorohod-Mitoma topology, see Appendices~\ref{subsec:distro} and \ref{subsec:Skorohod-Mitoma}. The joint probability and expectation of the processes will be denoted by $\PP,\EE$, and their respective generators by $\Q\super{n}$, acting on a subset $\Dom \Q\super{n}$ of the Borel-measurable functions. We tacitly assume that the linear operators $\langle\phi,\cdot\rangle: z \mapsto\langle\phi,z\rangle$ are in $\Dom\Q\super{n}$ for all $n\in\NN$ and $\phi\in C^\infty(\TT^d)$ (or a subset thereof), where $\langle\cdot,\cdot \rangle:={}_{C^\infty(\TT^d)}\langle \cdot, \cdot \rangle_{C^\infty(\TT^d)^\ast}$ denotes the duality paring between $C^\infty(\TT^d)$ and $C^\infty(\TT^d)^\ast$, see Subsection~\ref{subsec:test functions}.

\begin{remark} We restrict to the compact torus to avoid working with rapidly decaying test functions and mass escaping to infinity, but the results in this section hold for any compact set. In fact, in the rest of this paper we shall apply these metatheorems to $\TT^d$, $\TT^d\times\{1,\hdots,d\}$ and $\TT^d\times\{1,\hdots,d\}\times\{+,-\}$ via trivial extension.
\end{remark}

\subsection{Deterministic hydrodynamic limits}
\label{subsec:meta det hydrodyn}

Numerous mathematical tools are available to show convergence of stochastic processes, for example the classic hydrodynamic limit techniques~\cite{KipnisLandim1999}, generator convergence~\cite{Liggett1985,ethier_kurtz_1986}, Mosco convergence of Dirichlet forms \cite{KOLESNIKOV2006382}, etc. We focus on the classic technique~\cite{KipnisLandim1999}. We first phrase the hydrodynamic limit theorem for the case when the limit is deterministic before moving on to stochastic limits in the next subsection.

The aim is to show that the sequence of martingales,
\begin{align}
   M\super{\phi,n}(t) &:=\langle \phi,Z\super{n}(t)\rangle - \langle \phi,Z\super{n}(0)\rangle - \int_0^t\!(\Q\super{n} \langle\phi,\cdot\rangle)(Z\super{n}(s))\,ds
\label{eq:meta martingale n}
\intertext{converges to}
  0&=\langle \phi,z(t)\rangle - \langle \phi,z^0\rangle - \int_0^t\!\big\langle \phi, b(z(s))\big\rangle\,ds.
\label{eq:meta deterministic martingale}
\end{align}
If we can show this for a sufficiently large class $\A$ of test functions $\phi$ such that \eqref{eq:meta deterministic martingale} characterises a unique deterministic limit path, then that limit is characterised by the equation:
\begin{align*}
  \partial_t z(t)=b\big(z(t)\big). && \text{(in mild, distributional sense)}
\end{align*}

\begin{theorem}[\cite{KipnisLandim1999}] Assume that the sequence of processes $(Z\super{n})_n$ is tight in $D(0,T;C^\infty(\TT^d)^*)$, and that $Z\super{n}(0)$ converges narrowly to a unique deterministic limit $z$ in $C^\infty(\TT^d)^*$. Let $\A\subset\big\{\phi\in C^\infty(\TT^d):\langle\phi,\cdot\rangle\in \bigcap_{n\geq1}\Dom\Q\super{n}\big\}$. Suppose that there exists an operator $b:C^\infty(\TT^d)^* \to C^\infty(\TT^d)^*$ such that for any arbitrary 
$\phi\in\A$, any accumulation point $Z\in D(0,T;C^\infty(\TT^d)^*)$ and any $t\in\lbrack0,T\rbrack$,
\begin{enumerate}[(i)]
\item $\EE\big\lbrack (\Q\super{n} \langle\phi,\cdot\rangle)(Z\super{n}(t))-\langle\phi,b(Z(t))\rangle \big\rbrack^2\to0$,
\item\label{assit:mart conv 0} $\EE\langle M\super{\phi,n}\rangle_T\to0$.
\end{enumerate}\vspace{-0.3em}
If there exists a unique solution $z$ such that \eqref{eq:meta deterministic martingale} holds for all $\phi\in\A$, then $Z\super{n} \to z$ narrowly in $D(0,T;C^\infty(\TT^d)^*)$. 
\label{th:hydrodyn limit det}
\end{theorem}
\begin{proof} The result follows from Theorem~\ref{th:hydrodyn limit} (although a more direct proof is mostly trivial). Part \eqref{assit:mart conv 0} only needs to be checked for the final time $T$ due to Doob's martingale inequality. 
\end{proof}

\subsection{Stochastic hydrodynamic limits}
\label{subsec:meta stoch hydrodyn}
We now phrase the hydrodynamic limit theorem in the case where the limiting object is stochastic. The goal is to show that the martingales~\eqref{eq:meta martingale n} converge to limiting martingales of the form
\begin{align}
  M^{\phi}(t) := \langle \phi,Z(t)\rangle - \langle \phi,Z^0\rangle - \int_0^t \!\big\langle\phi,b(Z(s))\big\rangle\,ds.
\label{eq:meta random martingale}
\end{align}
By characterising the quadratic variation of \( M^\phi(t) \), we can represent it in terms of a weighted space-time white noise, yielding the SPDE in weak form
\begin{align}
  \langle \phi,Z(t)\rangle - \langle \phi,Z^0\rangle = \int_0^t\! \big\langle\phi,b(Z(s))\big\rangle\,ds + \llangle \phi*\sigma ,\Xi  \rrangle,
\label{eq:meta weak SPDE}
\end{align}
where \( \Xi \in C^\infty([0,T] \times \TT^d)^\ast \) denotes space-time white noise, the double bracket $\llangle\cdot,\cdot\rrangle$ denotes the pairing between time-dependent test functions and distributions (see Appendix~\ref{sec:topo}), and for some \( \sigma\colon [0,T] \times \TT^d \to C^\infty(\TT^d)^\ast \) (see Theorem~\ref{th:hydrodyn limit} below) and a test function \( \phi \in C^\infty(\TT^d)\) we define the function $\phi*\sigma$ by
\begin{equation*}
  (\phi*\sigma)(x,t) := \langle \phi, \sigma(x,t) \rangle.
\end{equation*}

If the class \( \mathcal{A} \) of test functions \( \phi \) (for which we can prove~\eqref{eq:meta random martingale}) is sufficiently rich, then the weak formulation~\eqref{eq:meta weak SPDE} can be formally rewritten as the stochastic partial differential equation
\begin{equation*}
  \partial_t Z(x,t) = b(Z(x,t)) + (\sigma * \Xi)(x,t),
\end{equation*}
where, for each fixed \( t \geq 0 \), the distribution \( (\sigma * \Xi)_t \in C^\infty(\TT^d)^* \) acts on test functions \( \phi \in \mathcal{A} \subset C(\TT^d) \) via
\begin{equation*}
  \langle \phi, \sigma*\Xi_t\rangle := \langle \phi*\sigma, \Xi_t \rangle.
\end{equation*}

\begin{theorem}[\cite{KipnisLandim1999}] 
Assume that the sequence of processes $(Z\super{n})_n$ is tight in $D(0,T;C^\infty(\TT^d)^*)$ such that any accumulation point is continuous, and that $Z\super{n}(0)$ converges narrowly to a unique (random) limit $Z^0$ in $C^\infty(\TT^d)^*$. Let $\A\subset\big\{\phi\in C^\infty(\TT^d):\langle\phi,\cdot\rangle\in \bigcap_{n\geq1}\Dom\Q\super{n}\big\}$. Suppose that there exist an operator $b:C^\infty(\TT^d)^* \to C^\infty(\TT^d)^*$ and a $ \sigma\colon [0,T] \times \TT^d \to C^\infty(\TT^d)^\ast $ such that for any arbitrary $\phi\in\A$, any accumulation point $Z\in D(0,T;C^\infty(\TT^d)^*)$, and any $t\in\lbrack0,T\rbrack$,
\begin{enumerate}[(i)]
\item\label{assit:gen conv 2} $\EE\big\lbrack(\Q\super{n} \langle\phi,\cdot\rangle)(Z\super{n}(t))-\langle\phi,b(Z(t))\rangle \big\rbrack^2\to0$,
\item\label{assit:mart conv 2} $\EE\big\lbrack \langle M\super{\phi,n}\rangle_T-\int_0^T\!\int_{\TT^d}\!(\phi*\sigma(x,t))^2\,dx\,dt \big\rbrack^2\to0$,
\item\label{assit:mart property linearfunc} There exists a positive sequence $c_n\to0$  such that
\[
\lim_{n \to \infty} \PP\Big\lbrack{ \textstyle \sup_{t \in [0,T]} \lvert \langle \phi,Z\super{n}(t)\rangle -\langle \phi,Z\super{n}(t^-)\rangle \rvert > c_n } \Big\rbrack = 0.
\]
\item\label{assit:mart property control kfield} For all $k \in \{ 1,2,3,4\}$ we have  
\[
\sup_{ n \in \NN} \, \sup_{ 0 \leq t \leq T} \EE \left[ \big\{(\Gamma_k\super{n} \langle \phi,\cdot\rangle)(Z\super{n}(t))\big\}^2 \right] < \infty.
\]
where $\Gamma_k\super{n}$ is given as in \eqref{eq:k_field}.
\end{enumerate}
Assume furthermore that $\sigma(t,\cdot) \in C^\infty(\TT^d; C^\infty(\TT^d)^*)$ for each $t \in [0,T]$, in the sense that for every $\phi \in C^\infty(\TT^d)$, the map $x \mapsto \langle \phi, \sigma(t,x) \rangle$ belongs to $C^\infty(\TT^d)$, and that $\langle \phi, \sigma(t,x) \rangle \in L^2(\TT^d \times [0,T])$.

If there exists a unique solution $Z$ such that \eqref{eq:meta weak SPDE} holds for all $\phi\in\A$, then $Z\super{n} \to Z$ narrowly in $D(0,T;C^\infty(\TT^d)^*)$.
\label{th:hydrodyn limit}
\end{theorem}

\begin{remark}
Using~\eqref{eq:Dynkin_martingale}, \eqref{eq:Dynkin_QuadVar} and Lemma~\ref{lem: quadvar of quadvar}, we can see that condition~\eqref{assit:mart property control kfield} implies: 
\begin{equation}\label{assit:mart property}
  \sup_{n \in \NN} \sup_{ 0 \leq t \leq T} \EE\lbrack M\super{\phi,n}(t)^2 \rbrack < \infty, 
\end{equation}
and
\begin{equation}\label{assit:mart property QuadVar} \sup_{n \in \NN} \sup_{ 0 \leq t \leq T} \EE\lbrack \int_0^T\!((M\super{\phi,n}(t))^2-\EE\langle M\super{\phi,n}\rangle_t)^2\,dt \rbrack < \infty.
\end{equation}
These two conditions are indeed necessary for the proof of Theorem~\ref{th:hydrodyn limit}.
\end{remark}

\begin{remark} Conditions \eqref{assit:gen conv 2} and \eqref{assit:mart conv 2} imply that 
\begin{align}
  &\EE\big\lbrack(\Q\super{n} \langle\phi,\cdot\rangle)(Z\super{n}(t))\big\rbrack\to\EE\big\lbrack \langle\phi,b(Z(t))\rangle  \big\rbrack, \label{assit:gen conv 1}\\
  &\EE\langle M\super{\phi,n}\rangle_t\to \int_0^t\!\int_{\TT^{d}}\! (\phi *\sigma(x,s))^2 \,dx\,ds \label{assit:mart conv 1},
\end{align}
which are used in practice to identify the candidates for $b$ and $\sigma$.
\label{rem:b and sigma}
\end{remark}

\begin{proof} Take a convergent subsequence with limit $Z(t)$. 

We first show term-by-term that the Dynkyn martingale~\eqref{eq:meta martingale n} converges to the limiting equation~\eqref{eq:meta random martingale}. Clearly the linear terms $\langle \phi,Z\super{n}(t)\rangle, \langle \phi,Z\super{n}(0)\rangle$ converge to $\langle \phi,Z(t)\rangle, \langle \phi,Z^0\rangle$ due to tightness respectively convergence of the initial condition. The generator part $\int_0^t\!(\Q\super{n} \langle\phi,\cdot\rangle)(Z\super{n}(s))\,ds$ converges to $\int_0^t\!\langle\phi,b(Z(s))\rangle\,ds$ by assumption~\eqref{assit:gen conv 2} and its implication~\eqref{assit:gen conv 1}, combined with Jensen's inequality to deal with the time integral.

Second, we show that $M\super{\phi,n}(t)$ retains the martingale property as $n\to\infty$. 
The aim is to use~\cite[Cor.~II.3.6]{Rebolledo1979}, by checking the two (stronger than needed) conditions:
\begin{itemize}
    \item The exists positive constants $c_n \to 0$ such that
    \begin{equation}\label{eq:meta jumps martingale Reb1}
        \PP\left( \sup_{t \in [0,T]} \lvert M\super{\phi,n}(t) -M\super{\phi,n}(t^-) \rvert > c_n \right) \to 0.
    \end{equation}
    \item There exists $g\colon \lbrack0,\infty) \to \lbrack0,\infty)$, with $g(u) \to 0$ as $u \to 0$, such that
    \begin{equation}\label{eq:meta martingale Reb2}
        \PP\left( \lvert \langle M\super{\phi,n}\rangle_t -\langle M\super{\phi,n}\rangle_s\rvert > g(|t-s|)\right) \to 0.
    \end{equation}
    for all $t,s \in [0,T]$.
\end{itemize}
 Condition \eqref{eq:meta jumps martingale Reb1} is an immediate consequence of \eqref{assit:mart property linearfunc}. To verify condition \eqref{eq:meta martingale Reb2}, we note that, using \eqref{eq:Dynkin_QuadVar}, for all $s,t \in [0,T]$,
\[
\lvert \langle M\super{\phi,n}\rangle_t -\langle M\super{\phi,n}\rangle_s \rvert  \leq  \lvert t-s \rvert \sup_{r \in [0,T]} \lvert \Gamma\super{n}_2 \langle\phi,\cdot\rangle)(Z\super{n}(r) ) \rvert,
\]
which for $\hat{\gamma}:= 1+ \sup_{n \in \NN} \sup_{r \in [0,T]} \lvert \Gamma\super{n}_2 \langle\phi,\cdot\rangle)(Z\super{n}(r) ) \rvert$, implies \eqref{eq:meta martingale Reb2}  if we choose $g(u):= \hat{\gamma} u$ and apply \eqref{assit:mart property control kfield} with $k=2$. Hence by~\cite[Cor.~II.3.6]{Rebolledo1979} any limit point $M^{\phi}$ of $M\super{\phi,n}$ is a continuous martingale which we can characterize by \eqref{eq:meta random martingale}. 

Third, we study properties of the quadratic variation of the limit $M^\phi(t)$. By assumptions \eqref{assit:mart conv 2} (and as a consequence~\eqref{assit:mart conv 1}), \eqref{assit:mart property}, and \eqref{assit:mart property QuadVar}, and a similar argument that the one we just described, this time applying \eqref{assit:mart property control kfield} with $k=2,3,4$ and the characterization given in \eqref{eq:Dynkin_QuadVar} plus Lemma~\ref{lem: quadvar of quadvar}, the limit $(M^\phi(t))^2-\int_0^t\!\int_{\TT^d}\!\big(\phi*\sigma(x,t)\big)^2\,dx\,dt$ is again a martingale. This implies that the quadratic variation is \cite[Chap.~IV,Th.~1.3]{RevuzYor1999}:
\begin{equation}
  \lbrack M^\phi\rbrack_t=\langle M^\phi\rangle_t=\int_0^t\!\int_{\TT^d}\!\big(\phi*\sigma(x,t)\big)^2\,dx\,dt,
\label{eq:tested QV}
\end{equation}
that is, not just in expectation, and this quadratic variation uniquely characterises the limiting martingale $M^\phi(t)$. 

We conclude that the martingale $M^\phi(t)$ is the same (in distribution) as the martingale $\langle \phi*\sigma_t,\Xi_t\rangle$ because it has precisely the same quadratic variation~ \cite[Prop.~2.2.6]{DalangSanzSole2024TR}.
\end{proof}

\begin{remark} Assumption~\eqref{assit:mart conv 2} includes the requirement that the limiting quadratic variation must be deterministic; this is essential for the argument.
\end{remark}

\begin{remark} Alternatively to \eqref{eq:meta weak SPDE} one can also derive the following SDE from the proof of Theorem~\ref{th:hydrodyn limit}. If $Z(t)$ is continuous then so is the martingale $M^\phi(t)$, and by the Dubins-Schwarz Theorem~\cite[Th.~1.6]{RevuzYor1999}, it can be replaced by a time-changed Brownian motion, where the time change is given by the (deterministic) quadratic variation \eqref{eq:tested QV}, leading to:
\begin{align*}
	d\langle \phi,Z(t)\rangle = \big\langle\phi,b(Z(t))\big\rangle\,dt + \sqrt{\textstyle \int_{\TT^d}\!\big(\phi*\sigma(x,t)\big)^2\,dx}\, dB^\phi(t).
\end{align*}
However, this construction yields different Brownian motions $B^\phi$ for each $\phi$, and there is little control on how these are correlated. In particular, it is difficult to exploit linearity in $\phi$ of this weak formulation.
\end{remark}

\subsection{Tightness criteria}
\label{subsec:meta tightness}

The basis of our tightness arguments is the following useful fact about the Skorohod-Mitoma topology.
\begin{theorem}[{\cite[Th.~4.1]{Mitoma1983}}]
If for all $\phi\in C^\infty(\TT^d)$ the sequence $(t\mapsto\langle \phi,Z\super{n}(t)\rangle)_{n\geq1}$ is tight in $D(0,T)$ (with the Skorohod J1-topology), then $(Z\super{n})_n$ has a narrowly Skorohod-Mitoma-convergent subsequence.
\label{th:mitoma}
\end{theorem}
\begin{proof}
By separability we may restrict to a countable number of test functions. Hence by Prohorov's Theorem for real-valued processes and a diagonal argument we can choose a subsequence such that all processes $\langle\phi,Z\super{n}\rangle_n$ converge narrowly in $D(0,T)$ to $\langle\phi,Z\rangle$ for some cluster point $Z$ that does not depend on the choice of $\phi$. Since we are only after sequential convergence, the Montel property of $C^\infty(\TT^d)^*$ (see Appendix~\ref{subsec:distro}) implies the convergence in the Skorohod-Mitoma topology.
\end{proof}
Therefore, in order to show tightness one can focus on real-valued processes. In fact, the following (fairly standard) tightness criteria apply to any real-valued processes, but to avoid introducing more notation we formulate everything in terms of $\langle \phi,Z\super{n}(t)\rangle$ for an arbitrary but fixed test function $\phi\in C^\infty(\TT^d)$.

Recall the usual carr{\'e} du champ operator, i.e. \eqref{eq:square_field}, and the modulus of continuity:
\begin{align*}
  \omega'(\langle\phi,z\rangle,\delta,T)&:=\inf_{\substack{0\leq t_0\hdots\leq t_K\leq T \\ t_i-t_{i-1}>\delta}} \max_{i=0,\hdots,K-1} \sup_{s,t\in\lbrack t_i,t_{i+1})} \lvert \langle \phi,z(t)-z(s)\rangle\rvert.
\end{align*}

\begin{lemma}[{\cite[Th.~13.2]{Billingsley1999}, \cite[Th.~4.1]{FengKurtz06} and \cite[Th.~2.3]{ferrari1988non}}] For any $\phi\in C^\infty(\TT^d)$:
\begin{enumerate}[(A)]
\item The sequence $(\langle \phi,Z\super{n}(t)\rangle)_{t\in\lbrack0,T\rbrack}$ is tight in $D(0,T)$ if and only if:
\begin{enumerate}[(i)]
\item $\lim_{\delta\to0}\limsup_{n\to\infty} \PP\big(\omega'(\langle\phi,Z\super{n}\rangle,\delta,T)\geq r\big)=0$ for all $r>0$,
\item for each $t\in\lbrack0,T\rbrack$, the sequence $\langle \phi,Z\super{n}(t)\rangle$ is tight.
\end{enumerate}

\item The sequence $(\langle \phi,Z\super{n}(t)\rangle)_{t\in\lbrack0,T\rbrack}$ is exponentially tight in $D(0,T)$ with speed $\beta_n$ if:
\begin{enumerate}[(i)]
\item $\lim_{\delta\to0}\limsup_{n\to\infty} \frac1{\beta_n}\log\PP\big(\omega'(\langle\phi,Z\super{n}\rangle,\delta,T)\geq r\big)=-\infty$ for all $r>0$,
\item for each $t\in\lbrack0,T\rbrack$, the sequence $\langle\phi,Z\super{n}(t)\rangle$ is exponentially tight.
\end{enumerate}

\item The sequence $(\langle \phi,Z\super{n}(t)\rangle)_{t\in\lbrack0,T\rbrack}$ is tight in $D(0,T)$ and has continuous limit points if,
\begin{enumerate}[(i)]
\item $\sup_{ n \in \NN} \, \sup_{ 0 \leq t \leq T} \EE \left[\langle \phi,Z\super{n}(t)\rangle^2 \right]< \infty$,
\label{it:ferrari bdd square}
\item $\sup_{ n \in \NN} \, \sup_{ 0 \leq t \leq T} \EE \left[\big\{(\Q\super{n} \langle \phi,\cdot\rangle)(Z\super{n}(t))\big\}^2\right] < \infty$,
\label{it:ferrari bdd generator}
\item $\sup_{ n \in \NN} \, \sup_{ 0 \leq t \leq T} \EE \left[ \big\{(\Gamma\super{n}_2 \langle \phi,\cdot\rangle)(Z\super{n}(t))\big\}^2 \right] < \infty$,
\label{it:ferrari bdd carre}
\item there exists a sequence $\delta_n\to0$ such that:
\begin{equation*}
    \lim_{ n \to \infty} \PP\Big({\textstyle \sup_{ 0 \leq t \leq T} \big\lvert \langle \phi,Z\super{n}(t)-Z\super{n}(t^-)\rangle \big\rvert \geq \delta_n} \Big) = 0.
\end{equation*}
\label{it:ferrari bdd jumps}
\end{enumerate}
\label{lem:tightness criteria ferrari}
\end{enumerate}
\label{lem:tightness criteria}
\end{lemma}
\vspace{-2em}
Statement (C) is stated without proof in \cite{ferrari1988non}, but it can be shown rather directly by using a combination of \cite[Cor.~7.4]{ethier_kurtz_1986} and Doob's inequality.


\section{Correlations}
\label{sec:correlation functions}

This section contains the technical heart of our paper: the asymptotic behaviour and quantitative estimates of correlation functions for the $\SEP$. This will be crucial for the tightness and hydrodynamic limits of Sections~\ref{sec:tightness} and \ref{sec:mainresults}.
\\
Our starting point is to use self-duality to express $k$-point correlation functions
at time $t>0$ in terms of the initial configuration and the transition
probabilities of $k$ dual particles evolving according to the ``stirring process''.
This yields an exact representation of $k$-point correlations in
Lemma~\ref{lem: kpts exact}. Combining this representation with the convergence
of the stirring process to independent Brownian motions 
and the slowly varying product initial condition,
Lemma~\ref{kpointscorr} yields the convergence and estimates of suitably rescaled
$k$-point correlations to products of the limiting profile. We then integrate these
correlation estimates against test functions and identify the corresponding
limits for the empirical density and for the nearest-neighbour occupation
measure~\eqref{eq:NN measure} in Corollary~\ref{cor:kptsconvintegral}, including the
dependence on the different scaling regimes and the onset of additional
fluctuation terms or blow-up. Finally, we extend
these results to time-integrated correlations: Lemma~\ref{lem:2cor estimate},
Lemma~\ref{lem:4cor estimate} and Corollaries~\ref{cor:2cor int}--\ref{cor:4cor int}
provide uniform bounds and convergence statements for one- and two-time
integrated correlations, which will be crucial for controlling quadratic
variations and for verifying the conditions of the meta results, Theorem~\ref{th:hydrodyn limit det} and~\ref{th:hydrodyn limit}.

The limit correlations can be described in terms of the time-evolved version of $\rho_0$, i.e.:
\begin{equation*}
    \rho(x,t) := \EE\left[ \rho_0(B(t))\right],
\end{equation*}
with Brownian motion $B(t)$ starting from $x$, or equivalently, the solution to the diffusion equation:
\begin{align}
    \partial_t\rho(x,t)=\Delta\rho(x,t),                    && \rho(x,0) = \rho_0(x).
\label{eq:rho}
\end{align}

The key to checking most conditions of the Hydrodynamic Limit Theorems~\ref{th:hydrodyn limit det} and \ref{th:hydrodyn limit} is convergence of the correlation functions, which is closely related to Assumption~\ref{ass:initial} with fixed initial density $\rho_0$. Of particular interest is the measure-valued number of occupied nearest neighbours:
\begin{align}
  \Lambda\super{n}(t):=\frac1{\epsilon^d n^2} \sum_{\substack{x,y\in\TT_\epsilon^d:\\ x\nnp{l} y}} \eta(x,t)\eta(y,t)\delta_{\frac{x+y}{2}}\mathds1_l,
\label{eq:NN measure}
\end{align}
where $x\nnp{l} y$ are nearest neighbours in positive direction as introduced in the notation Subsection~\ref{subsec:notation}.

\subsection{Convergence of correlations}

We first show the convergence of the correlations and then of the correlations paired with test functions. It will be useful to allow small perturbations of the type, for any $x\in\TT^d$:
\begin{align*}
  x_\epsilon:=\epsilon\lfloor \tfrac1\epsilon x\rfloor\in \TT_\epsilon^d.
\end{align*}


The following lemma states that the $k$-point correlation function at time $t$ can be expressed by tracking the positions of $k$ dual particles (moving according to exclusion interaction) from time $0$ to time $t$, and then evaluating the product of initial occupation variables at these final positions.

\begin{lemma}[$k$-point correlations via self-duality]\label{lem: kpts exact}
For all $k\in\NN$, distinct points $x^1,\ldots x^k \in\TT^{d}_\epsilon$ and $t\in\lbrack0,T\rbrack$ we have
\begin{equation}
  \EE \left[ \prod_{i=1}^k \etaV[x^i][t] \right]= \sum_{ y \in \TT^{kd}_\epsilon} p_t\super{\STIR,k,\epsilon}( y \mid x) \, \EE \left[ \prod_{i=1}^k \etaV[y^i][0] \right], 
\label{eq: exact kcorr}
\end{equation}
where $p_t\super{\STIR,k,\epsilon}( y \mid x)$ denotes the transition probability of the $k$ particles ``stirring'' process with generator:
\begin{multline*}
  L^{(\STIR,k,\epsilon)} f(x) = \frac1{\epsilon^2} \sum_{i=1}^k \sum_{\substack{y\sim x^i\\y\neq x^j\,\forall j}} \left[ f(x^1,\hdots,y,\hdots,x^k) - f(x^1,\hdots,x^i,\hdots,x^k) \right] \\ 
  + \frac1{\epsilon^2} \sum_{\substack{i,j=1\\ x^i\sim x^j }}^k  \left[ f(x^1,\hdots,x^j,\hdots,x^i,\hdots,x^k) - f(x^1,\hdots,x^i,\hdots,x^j,\hdots,x^k) \right].
\end{multline*}
As a consequence, under Assumption~\ref{ass:initial},
\begin{align}
  \EE \left[ \prod_{i=1}^k \etaV[x^i][t] \right] \leq \big( \epsilon^d n \lVert\rho_0\rVert_C \big)^k.
\label{eq:k-correlation bound}
\end{align}
\end{lemma}

The above-mentioned stirring process describes the motion of labelled particles that perform continuous-time random walks on the lattice where, unlike in the symmetric exclusion process which forbids jumps to occupied sites, particles exchange their labels when they attempt to move to an occupied position~\cite{demasi2006mathematical}.

\begin{proof}
The $\SEP$ is known \cite[Chap. VIII]{Liggett1985} to satisfy a self-duality relation with self-duality function
\begin{equation}
    D(\eta_1, \eta_2) = \prod_{\substack{ x \in \TT^d_\epsilon \\ \eta_2(x)=1}} \eta_1(x).
\end{equation}
In particular, we can write
\begin{equation}
 \prod_{i=1}^k \etaV[x^i][t] = D\big({\textstyle \etaV[\cdot][t], \sum_{i=1}^k \delta_{x^i}}\big),   
\end{equation}
where $\sum_{i=1}^k \delta_{x^i}$ represents the configuration with exactly $k$ particles, one at each position $x^i$ for $i=1,\ldots,k$. This expression works because the self-duality function will select exactly the occupation variables at positions $x^1, \ldots, x^k$.

Let $\mu_{\rho_0}\in\P(\{0,1\}^{\TT_\epsilon^d})$ be the probability of the initial configuration from Assumption~\ref{ass:initial}, and $\EE_{\eta(0)}$ the expectation of the path, conditioned on initial configuration $\eta(0)$. By self-duality we have
\begin{align*}
\EE \left[ \prod_{i=1}^k \etaV[x^i][t] \right] &= \int \EE_{\eta(0)} \left[ D\big({\textstyle \etaV[\cdot][t], \sum_{i=1}^k \delta_{x^i}}\big) \right] \, \mu_{\rho_0}(d \eta(0)) \\
&=\int \EE_{\sum_{i=1}^k \delta_{x^i}} \left[ D\big({\textstyle\eta(0), \sum_{i=1}^k \delta_{x^i(t)}}\big) \right] \, \mu_{\rho_0}(d \eta(0)) \\
&= \EE_{\sum_{i=1}^k \delta_{x^i}} \int  \prod_{i=1}^k \etaV[x^i(t)][0]  \, \mu_{\rho_0}(d \eta(0)) \\
&= \sum_{ y \in \TT^{kd}_\epsilon} p_t\super{\STIR,k,\epsilon}( y \mid x) \, \EE \left[ \prod_{i=1}^k \etaV[y^i][0] \right],
\end{align*}
where $x^i(t)$ represents the position at time $t$ of the particle initially at position $x^i$. This proves the exact expression~\eqref{eq: exact kcorr}.

The bound~\eqref{eq:k-correlation bound} follows as a simple consequence of the independent Bernoulli distributions at time $0$:
\begin{align*}
  \EE \left[ \prod_{i=1}^k \etaV[x^i][t] \right] 
  &= \sum_{ y \in \TT^{kd}_\epsilon} p_t\super{\STIR,k,\epsilon}( y \mid x) \, 
  \prod_{i=1}^k \Big( n\int_{B_\epsilon(y^i)}\rho_0(z^i)\,dz^i \Big)\\
  &\leq \underbrace{\sum_{ y \in \TT^{kd}_\epsilon} p_t\super{\STIR,k,\epsilon}( y \mid x) \, 
  }_{=1}\prod_{i=1}^k \Big( n\epsilon^d \lVert\rho_0\rVert_C \Big).
\end{align*}
\end{proof}


\begin{lemma}[$k$-points correlations]\label{kpointscorr} Under Assumption~\ref{ass:initial}, for all $k\in\NN$, distinct $x^1,\ldots x^k \in\TT^d$ and $t\in\lbrack0,T\rbrack$,
\begin{equation}
\frac1{(\epsilon^d n)^k} \, \EE \left[ \prod_{i=1}^k \etaV[x_\epsilon^i][t] \right]  \xrightarrow{n\to\infty} \prod_{i=1}^k \rho(x^i,t). 
\label{eq:k-point correlations raw}
\end{equation}
\label{lem:correlation functions}
\end{lemma}

\begin{proof}
Note that by Lemma~\ref{lem: kpts exact} we have
 \begin{align}\label{EstimateDiff}
&\left| \frac1{(\epsilon^d n)^k} \, \EE \left[ \prod_{i=1}^k \etaV[x_\epsilon^i][t] \right]  - \prod_{i=1}^k \rho(x^i,t)  \right| \nonumber \\
&\qquad\leq    \sum_{ y_\epsilon \in \TT^{kd}_\epsilon}   \left| p_t\super{\STIR,k,\epsilon}( y_\epsilon \mid x_\epsilon)-p_{\epsilon^{-2}t}^{(\BM,k)}( y_\epsilon \mid x_\epsilon) \right| \frac1{(\epsilon^d n)^k} \, \EE \left[ \prod_{i=1}^k \etaV[y_\epsilon^i][0] \right]\nonumber \\
&\qquad\qquad+ \sum_{y^1_\epsilon,y^2_\epsilon} p_{\epsilon^{-2}t}^{(\BM,k)}( y_\epsilon \mid x_\epsilon) \left|\frac1{(\epsilon^d n)^k} \, \EE \left[ \prod_{i=1}^k \etaV[y_\epsilon^i][0] \right]- \prod_{i=1}^k \rho_0(y_\epsilon^i)\right| \nonumber \\
&\qquad\qquad+\left| \sum_{y^1_\epsilon,y^2_\epsilon} p_{\epsilon^{-2}t}^{(\BM,k)}( y_\epsilon \mid x_\epsilon)\, \prod_{i=1}^k \rho_0(y_\epsilon^i)-\prod_{i=1}^k \rho(x^i,t)\right|,
\end{align}
where $p_{\epsilon^{-2}t}^{(\BM,k)}( y_\epsilon \mid x_\epsilon)$ denotes the transition probability of $k$ independent Brownian particles at time $\epsilon^{-2}t$. 

We estimate each term on the right-hand side of \eqref{EstimateDiff} separately. For the first term, we use that the stirring process converges to independent Brownian motions in the following sense~\cite[Prop.~6.3 \& Eq.~(6.34a)]{demasi2006mathematical}:
\begin{equation}\label{Demasi difference}
  \lim_{\epsilon \to 0 }  \sup_{x_\epsilon}  \sum_{ y_\epsilon \in \TT^{kd}_\epsilon}   \left| p_t\super{\STIR,k,\epsilon}( y_\epsilon \mid x_\epsilon)-p_{\epsilon^{-2}t}^{(\BM,k)}( y_\epsilon \mid x_\epsilon) \right| = 0.
\end{equation}
Assumption~\ref{ass:initial} implies that the second term of \eqref{EstimateDiff} vanishes. For the last term we use that 
\begin{equation*}
\left| p_{\epsilon^{-2}t}^{(\BM,k)}( y_\epsilon \mid x_\epsilon) - \epsilon^{kd} p_{t}^{(\BM,k)}( y_\epsilon \mid x_\epsilon) \right|
= \mathcal{O}(\epsilon^{kd+2}).
\end{equation*}
\end{proof}

Clearly, when only $l$ of the $k$ points are distinct, then the product $\prod_{i=1}^k \eta(x_\epsilon^i,t)$ in \eqref{eq:k-point correlations raw} collapses to $l$ factors. This observation plays a subtle role in the next lemma. Recall the nearest-neighbour measure~\eqref{eq:NN measure}, and the diffusion kernel~\eqref{eq:rho}, as well as the different bracket notations~\eqref{eq:brackets notations}.

\begin{corollary}[$k$-points integrated correlations]
\begin{align*}
\intertext{one-point correlation, for all $\phi \in C(\TT^d)$:}
  &\EE\langle\phi,\Rho\super{n}(t)\rangle_{x} \to  \langle\phi,\rho(t)\rangle_{x},
\intertext{two-points correlation, for all $\phi \in C(\TT^d\times\{1,\hdots,d\})$:} 
  &\EE\langle\phi,\Lambda\super{n}(t)\rangle_{x,l} \to \langle \phi,\rho^2(t)\rangle_{x,l},
\intertext{three-points correlation, for all $\phi \in C(\TT^d)$  and  $\psi \in C(\TT^d\times\{1,\hdots,d\})$:} 
  &\epsilon^d n \, \EE \langle\phi,\Rho\super{n}(t)\rangle_{x} \langle\psi,\Lambda\super{n}(t)\rangle_{x,l} \to 
    \begin{cases}
    \langle \phi,\rho(t)\rangle_{x} \langle \psi,\rho^2(t)\rangle_{x,l}, &\frac1{\epsilon^{d}}=n,\\  
    0, &\frac1{\epsilon^{d/2}}\ll n \ll \frac1{\epsilon^d},
  \end{cases}
\intertext{four-points correlation, for all $\phi \in C(\TT^d\times\{1,\hdots,d\})$:} 
  &\EE\langle\phi,\Lambda\super{n}(t)\rangle_{x,l}^2 \to
  \begin{cases}
    \langle \phi,\rho^2(t)\rangle_{x,l}^2, &\frac1{\epsilon^{d/2}}\ll n \lle \frac1{\epsilon^d},\\
    \langle \phi,\rho^2(t)\rangle_{x,l}^2 + \langle \phi^2,\rho^2(t)\rangle_{x,l}, &\frac1{\epsilon^{d/2}}=n,   
  \end{cases}
\end{align*}
and the four-points correlation blows up if $\frac1{\epsilon^{d/2}}\gg n$ and $\supp (\sum_l\phi_l)\cap\supp \rho_0\neq\emptyset$.
\label{cor:kptsconvintegral}
\end{corollary}

\begin{proof}\phantom{a}
\begin{itemize}
\item One-point correlation: This is proven analogously to the two-point correlations.
\item Two-point correlation:
We use Lemma~\ref{kpointscorr} to replace the rescaled expectations by the solution to the hydrodynamic equation. It will be helpful to abbreviate $y:=x+\epsilon\mathds1_l$, that is, $y$ is a nearest neighbour of $x$ in the positive direction of dimension $l$. In particular, $x$ and $y$ are distinct.
\begin{align*}
  \EE\langle\phi,\Lambda\super{n}(t)\rangle_{x,l}=
  &\frac1{\epsilon^d n^2} \sum_{x\nnp{l}y} \EE \etaV[x][t]\etaV[y][t] \phi_l(\tfrac{x+y}{2}) \\
  &= \epsilon^d \sum_{x\nnp{l}y}\sum_{l=1}^d \rho(x,t)\rho(y,t)\, \phi_l(\tfrac{x+y}{2}) + o(1)\\
  &= \epsilon^d \sum_{x\nnp{l}y}\sum_{l=1}^d \rho^2(x,t)\, \phi_l(x) + o(1),
\end{align*}
by the uniform continuity of $\phi$ and $\rho(t)$. This converges to the claimed limit
\begin{align*}
  \sum_{l=1}^d \int\!\rho^2(x,t)\, \phi_l\!\left(x\right)\,dx=\langle\phi,\rho^2\rangle_{x,l}.
\end{align*}

\item Four-points correlations: we only show the result for $t=0$; the propagation to later times is done analogously to the two-point correlation above. Recall that initially $(\eta(x,0))_{x\in\TT_\epsilon^d}$ has independent Bernoulli distributions with probability $n\int_{B_\epsilon(x)}\!\rho_0(y)\,dy=\epsilon^d n\rho_0(x)+o(1)$. However, in order to exploit this indepence, we need to distinguish how many distinct points appear in the sum:
\begin{align}
  &\EE\langle\phi,\Lambda\super{n}(0)\rangle_{x,l}^2 \nonumber \\
  &=\frac{1}{\epsilon^{2d}n^4}\sum_{x^1\nnp{l_1} y^1} \sum_{x^2\nnp{l_2}y^2}
    \phi_{l_1}(\tfrac{x^1+y^1}{2}) \phi_{l_2}(\tfrac{x^2+y^2}{2}) 
  \EE \eta(x^1,0)\eta(y^1,0) \eta(x^2,0)\eta(y^2,0) \nonumber \\
  &=\frac{1}{\epsilon^{2d}n^4}\sumsum_{4\text{ distinct points}}\hdots
   +\frac{1}{\epsilon^{2d}n^4}\sumsum_{3\text{ distinct points}}\hdots
   +\frac{1}{\epsilon^{2d}n^4}\sumsum_{2\text{ distinct points}}\hdots \nonumber \\
  &=\frac{1}{\epsilon^{2d}n^4}\sumsum_{\substack{x^1\nnp{l_1} y^1, x^2\nnp{l_2}y^2\\4\text{ distinct points}}}\phi_{l^1}(x^1) \phi_{l^2}(x^2) 
  (\epsilon^d n)^4 \rho_0(x^1)\rho_0(y^1)\rho_0(x^2)\rho_0(y^2) \nonumber \\
  &\hspace{10em}+\frac{1}{\epsilon^{2d}n^4}\sum_{\substack{x\nnp{l} y\\\text{(2 distinct points)}}} \phi^2_{l}(x) (\epsilon^d n)^2\rho_0^2(x) + o(1),
  \label{eq.1.corrt0.splitting}
\end{align}
where in the last step we used the explicit initial condition together with uniform continuity as above, and the fact that the sum over 3 distinct points is of order $\bigoh(\tfrac1n)$ and can thus be included in the term $o(1)$. Similarly, the first sum over 4 distinct points is $\bigoh(1)$, and the last sum over two distinct points is $\bigoh(\tfrac1{\epsilon^d n^2})$. It follows that, as claimed,
\begin{multline*}
  \EE\langle\phi,\Lambda\super{n}(0)\rangle_{x,l}^2 \\ \to
  \begin{cases}
    \Big({\textstyle \sum_{l=1}^d \int\!\phi_l(x)\rho_0^2(x)\,dx}\Big)^2, 
      &\frac1{\epsilon^{d/2}}\ll n \lle \frac1{\epsilon^d},\\
    \Big({\textstyle \sum_{l=1}^d \int\!\phi_l(x)\rho_0^2(x)\,dx}\Big)^2 + \sum_{l=1}^d \int\!\phi_l^2(x)\rho_0^2(x)\,dx,
      &\frac1{\epsilon^{d/2}}=n.
  \end{cases}  
\end{multline*}
Moreover, if $\epsilon^d n^2 \ll 1$ and there exists a compact non-null set $K\subset\TT^d$ on which $\kappa=\inf_K \sum_l \phi_l^2\rho_0^2>0$, then the sum over two distinct points is
\begin{align*}
  \frac{1}{n^2}\sum_{x\nnp{l}y} \phi^2_{l}(x) \rho_0^2(x) \geq \frac{\kappa}{n^2} \#K\cap\TT_\epsilon^d\geq\frac{\kappa \lvert K\rvert}{2\epsilon^d n^2}\to\infty.
\end{align*}

\item Three-points correlations: as in the four-points case, we work at $t=0$; propagation to $t>0$ uses Lemma~\ref{kpointscorr} exactly as before. Let
\begin{align*}
  \epsilon^d n \,\EE\langle\phi,\Rho^n(0)\rangle_x\,
                 \langle\psi,\Lambda^n(0)\rangle_{x,l}
&=\frac1{n^2}\sum_{u}\sum_{x\nnp{l}y}
\EE[\eta(u)\eta(x)\eta(y)]\,\phi(u)\psi_l\!\Big(\tfrac{x+y}{2}\Big) \\
&=
\underbrace{\frac1{n^2}\sumsum_{2\text{ distinct points}}\hdots}_{=:A_n\super{2}}
   +\underbrace{\frac1{n^2}\sumsum_{3\text{ distinct points}}\hdots}_{=:A_n\super{3}}.
\end{align*}
The key observation is that:  
\[
\EE[\eta(u)\eta(x)\eta(y)] 
\begin{cases}
=(\epsilon^dn)^3\rho_0(u)\rho_0(x)\rho_0(y)+o((\epsilon^dn)^3), &\text{3 distinct points},\\[0.2em]
\le C(\epsilon^dn)^2, &\text{2 distinct points},
\end{cases}
\]
and the number of triples is  
\[
\#\{u,x,y\text{ all distinct}\}= \bigoh(\epsilon^{-2d}),\qquad
\#\{u=x\neq y\text{ or }u=y\neq x\}=\bigoh(\epsilon^{-d}).
\]

\medskip
Using uniform continuity to replace 
$\rho_0(y)\mapsto\rho_0(x)$ and $\psi_l((x+y)/2)\mapsto\psi_l(x)$,
\begin{equation*}
A_n^{(3)}
= \frac{(\epsilon^dn)^3}{n^2}
\sum_{u,x\nnp{l}y}\rho_0(u)\rho_0^2(x)\phi(u)\psi_l(x)+o(1)
= (\epsilon^dn)\,\langle\phi,\rho_0\rangle_x\,
                 \langle\psi,\rho_0^2\rangle_{x,l}+o(1).
\end{equation*}
and for two distinct points
\[
|A_n^{(2)}|
\;\lesssim\; \frac1{n^2}\,\epsilon^{-d}(\epsilon^dn)^2
= \epsilon^d \xrightarrow[\epsilon\to0]{} 0.
\]
The claim follows from this.
\end{itemize}
\end{proof}

\subsection{Estimates and convergence of time-integrated correlations}
The results of Corollary~\ref{cor:kptsconvintegral} need to be extended to
time-integrated correlations, because later in Section~\ref{sec:tightness} and Section~\ref{sec:mainresults} we need to deal with additive functionals of the form
$\int_0^T\!\langle\phi,\Lambda^{(n)}(t)\rangle\,dt$, which enter the
carré--du--champ and quadratic variation of the Dynkin martingales in Theorems~\ref{th:hydrodyn limit det} and
\ref{th:hydrodyn limit}. In particular, we require uniform (in $n$) bounds and
identification of the limit of these time-integrated correlations to control
the martingale terms and to prove tightness of the processes involved. For the
two-point correlation this extension is rather straightforward, since it
involves expectations over two distinct sites.
\begin{lemma}
Under Assumption~\ref{ass:initial},
for all $\phi \in C(\TT^d\times\{1,\hdots,d\})$, $t \in [0, T]$ and $n\in\NN$,  
\begin{equation*}
  \EE\langle\phi,\Lambda\super{n}(t)\rangle_{x,l} \leq d  \lVert\phi\rVert_C \lVert\rho_0\rVert_C^2,
\end{equation*}
where $\lVert\cdot \rVert_C$ is the standard norm in $C(\TT^d\times\{1,\hdots,d\})$.
\label{lem:2cor estimate}
\end{lemma}
\begin{proof}
By Lemma~\ref{lem: kpts exact}:
\begin{align*}
 \EE\langle\phi,\Lambda\super{n}(t)\rangle_{x,l}
 \leq
  \epsilon^d  \lVert\phi\rVert_C \lVert\rho_0\rVert_C^2 \sum_{x\leadsto y} 1 = d  \lVert\phi\rVert_C \lVert\rho_0\rVert_C^2,
\end{align*}
as the sum runs over $1/\epsilon^d$ possible $x\in\TT_\epsilon^d$ and $d$ possible nearest neighbours $y$ in positive direction. 

\end{proof}

\begin{corollary} Under Assumption~\ref{ass:initial},
for all $\phi \in C(\TT^d\times\{1,\hdots,d\})$,
\begin{align*}
  \EE{\textstyle \int_0^T\!\langle\phi,\Lambda\super{n}(t)\rangle_{x,l}\,dt} &\leq d T \lVert\phi\rVert_C \lVert\rho_0\rVert_C^2,\\
  \EE{\textstyle \int_0^T\!\langle\phi,\Lambda\super{n}(t)\rangle_{x,l}\,dt} &\to \int_0^T\! \langle\phi,\rho^2(t)\rangle_{x,l}\,dt.
\end{align*}
\label{cor:2cor int}
\end{corollary}

For the four-points correlation we study $\EE(\int_0^T\!\langle\phi,\Lambda\super{n}(t)\rangle_{x,l}\,dt)^2$, which require estimates at two different times. This is more involved than the two-points correlation, because it involves the expectation over four, possibly not distinct points.
\begin{lemma}
Under Assumption~\ref{ass:initial},
for all $\phi,\psi \in C(\TT^d\times\{1,\hdots,d\})$, $0\leq t_1 \leq< t_2 \leq T$, 
\begin{multline*}
  \EE\langle\phi,\Lambda\super{n}(t_1)\rangle_{x,l} \,\langle\psi,\Lambda\super{n}(t_2)\rangle_{x,l} 
  \leq \lVert\phi\rVert_C\lVert\psi\rVert_C \Big(
    d^2\lVert\rho_0\rVert_C^4
    +
    \mfrac{4d}{n}  \lVert \rho_0\rVert_C^3
    +
    \mfrac{d}{\epsilon^d n^2} \lVert \rho_0\rVert_C^2 \Big).
\end{multline*}
\label{lem:4cor estimate}
\end{lemma}
\begin{proof} Expanding the product yields:
\begin{multline*}
  \EE\big\langle \phi,\Lambda\super{n}(t_1)\big\rangle_{x,l} \big\langle \psi,\Lambda\super{n}(t_2)\big\rangle_{x,l}\\
  = \frac{1}{\epsilon^{2d}n^4} \sum_{x^1\nnp{l_1}y^1} \sum_{x^2\nnp{l_2} y^2} 
      \phi_{l_1}(\tfrac{x^1+y^1}{2},t_1)
      \psi_{l_2}(\tfrac{x^2+y^2}{2},t_2)\\
  \EE\big\lbrack
        \eta(x^1,t_1)\eta(y^1,t_1)
        \eta(x^2,t_2)\eta(y^2,t_2)
     \big\rbrack
\end{multline*}
In order to deal with the separate times we now condition on the filtration at time \(t\).  By the Markov property and Lemma~\ref{lem: kpts exact}, for each pair $x^2\leadsto y^2$,
\[
\EE\!\bigl[\eta(x^2,t_2)\,\eta(y^2,t_2)\,\big|\;\eta(t_1)\bigr]
=\sum_{x^3\neq y^3}p\super{\STIR,2,\epsilon}_{t_2-t_1}(x^3,y^3\mid x^2,y^2)\,
\eta(x^3,t_1)\,\eta(y^3,t_1),
\]
where we may assume without loss of generality that $x^3\neq y^3$ because the stirring process does not allow two particles to occupy the same site. Keep in mind that $x^1\leadsto y^1$ and are nearest neighbours (in positive direction), and so are $x^2\leadsto y^2$, but $x^3\neq y^3$ are just distinct:
\begin{center}
\begin{tabular}{|l|cc|l}
\cline{1-3}
$t_2$ & & \tikzmark{tm pointstwo}$x^2\leadsto y^2$ \\
&&$\big\uparrow$ &  (stirring semigroup)\\
$t_1$ &  \tikzmark{tm pointsthree}$x^1\leadsto y^1$ & $x^3\neq y^3$\\
\cline{1-3}
\end{tabular}
\end{center}

We thus arrive at the (exact!) calculation:
\begin{align*}
  &\EE\big\langle \phi(t_1),\Lambda\super{n}(t_1)\big\rangle_{x,l} \big\langle \psi(t_2),\Lambda\super{n}(t_2)\big\rangle_{x,l}\\
  &=\frac{1}{\epsilon^{2d}n^4} \sum_{x^1\nnp{l_1}y^1}\sum_{x^2\nnp{l_2} y^2}\sum_{x^3\neq y^3} 
      \phi_{l_1}(\tfrac{x^1+y^1}{2},t_1)
      \psi_{l_2}(\tfrac{x^2+y^2}{2},t_2)\\
  &\hspace{8em}
    p\super{\STIR,2,\epsilon}_{t_2-t_1}(x^3,y^3\mid x^2,y^2)
    \EE\big\lbrack
        \eta(x^1,t_1)\eta(y^1,t_1)
        \eta(x^3,t_1)\eta(y^3,t_1)
     \big\rbrack\\
   &=\frac{1}{\epsilon^{2d}n^4} \sumsumsum_{x^1,y^1,x^3,y^3 \text{ all distinct}}\!\!\!\hdots  
     +\frac{1}{\epsilon^{2d}n^4} \sumsumsum_{3 \text{ distinct points}}\!\!\!\hdots
     +\frac{1}{\epsilon^{2d}n^4} \sumsumsum_{2 \text{ distinct points}}\!\!\!\hdots     
\end{align*}
where, similar to the proof of Corollary~\ref{cor:kptsconvintegral}, we need to distinguish between sums that run over 4, 3 or 2 distinct points in order to invoke Lemma~\ref{lem: kpts exact}.
We treat each case separately.
\begin{enumerate}[(I)]
\item If $x^1,y^1,x^3,y^3$ are all distinct, Lemma~\ref{lem: kpts exact} applies directly to the four-point expectation. If we pull out the test functions $\phi,\psi$ and use the bound~\eqref{eq:k-correlation bound}:
\begin{align*}
  &\frac{1}{\epsilon^{2d}n^4} \sumsumsum_{x^1,y^1,x^3,y^3 \text{ all distinct}}\!\!\!\hdots\\
  &\qquad\leq 
  \epsilon^{2d} \lVert\phi\rVert_C\lVert\psi\rVert_C \lVert \rho_0\rVert_C^4 \sum_{x^1\nnp{l_1}y^1}\sum_{x^2\nnp{l_2} y^2}\underbrace{\sum_{x^3\neq y^3} p\super{\STIR,2,\epsilon}_{t_2-t_1}(x^3,y^3\mid x^2,y^2)}_{=1}\\
  &\qquad= d^2\lVert\phi\rVert_C\lVert\psi\rVert_C \lVert \rho_0\rVert_C^4.
\end{align*}

\item For 3 distinct points: if $x^3\neq y^3$ are given, the following $x^1,y^1$ are possible:\\[-3em]
\begin{center}
\begin{tabular}{r|c|c|c|c|}
\hline
&$x^3=x^1\leadsto y^1$ & $y^3=x^1\leadsto y^1$ & $x^1\leadsto y^1=x^3$ & $x^1\leadsto y^1=y^3$ \\
&$y^1\neq y^3$         & $y^1\neq x^3$         & $x^1\neq y^3$         & $x^1\neq x^3$\\
\hline
\# cases & $\leq d$ & $\leq d$ & $\leq d$ & $\leq d$\\
\hline
\end{tabular}
\end{center}
Hence we arrive at:
\begin{align*}
  &\frac{1}{\epsilon^{2d}n^4} \sumsumsum_{\text{3 distinct points}}\!\!\!\hdots\\
  &\qquad\leq 
    \frac{\epsilon^d}{n} \lVert\phi\rVert_C\lVert\psi\rVert_C \lVert \rho_0\rVert_C^3 \sum_{x^2\nnp{l_2} y^2} \underbrace{\sum_{x^3\neq y^3} p\super{\STIR,2,\epsilon}_{t_2-t_1}(x^3,y^3\mid x^2,y^2)}_{=1} \underbrace{\sum_{\substack{x^1\nnp{l_1}y^1\\ \text{3 distinct points}}}1}_{\leq 4d}\\[-1em]
  &\qquad\leq \frac{4d}{n} \lVert\phi\rVert_C\lVert\psi\rVert_C \lVert \rho_0\rVert_C^3. 
\end{align*}

\item For 2 distinct points there are only two possibilities: either $(x^1,y^1)=(x^3,y^3)$ or $(x^1,y^1)=(y^3,x^3)$. After pulling out $\phi,\psi$ and applying \eqref{eq:k-correlation bound}, we then use a very crude bound, which turns out to be of the correct order:
\begin{align*}
  &\frac{1}{\epsilon^{2d}n^4} \sumsum_{2 \text{ distinct points}}\!\!\!\hdots\\
  &\qquad\leq \frac{1}{n^2} \lVert\phi\rVert_C\lVert\psi\rVert_C \lVert \rho_0\rVert_C^2 \\
  &\hspace{4em} \sum_{x^1\leadsto y^1} \sum_{x^2\leadsto y^2} \Big(p\super{\STIR,2,\epsilon}_{t_2-t_1}(x^1,y^1\mid x^2,y^2)+ p\super{\STIR,2,\epsilon}_{t_2-t_1}(y^1,x^1\mid x^2,y^2)\Big)\\
  &\qquad\leq \frac{1}{n^2} \lVert\phi\rVert_C\lVert\psi\rVert_C \lVert \rho_0\rVert_C^2 
     \sum_{x^2\leadsto y^2} \underbrace{\sum_{x^1\neq y^1} p\super{\STIR,2,\epsilon}_{t_2-t_1}(x^1,y^1\mid x^2,y^2) }_{=1}\\[-0.5em]
  &\qquad= \frac{d}{\epsilon^d n^2} \lVert\phi\rVert_C\lVert\psi\rVert_C \lVert \rho_0\rVert_C^2.
\end{align*}
\end{enumerate}

\end{proof}

\begin{corollary} Under Assumption~\ref{ass:initial},
for all $\phi \in C(\TT^d\times\{1,\hdots,d\})$,
\begin{align*}
  \EE\big({\textstyle \int_0^T\!\langle\phi,\Lambda\super{n}(t)\rangle_{x,l}\,dt}\big)^2
  \leq
  T^2\lVert\phi\rVert_C^2 \Big(
    d^2\lVert\rho_0\rVert_C^4
    +
    \mfrac{4d}{n}  \lVert \rho_0\rVert_C^3
    +
    \mfrac{d}{\epsilon^d n^2} \lVert \rho_0\rVert_C^2 \Big).
\end{align*}
\label{cor:4cor int}
\end{corollary}

\begin{corollary}
Let initial Assumptions~\ref{ass:initial} hold. Then in the scaling regime
\begin{align*}
  \frac1{\epsilon^{d/2}} \ll n \lle \frac{1}{\epsilon^d},
\end{align*}
and for arbitrary $\phi\in C(\TT^d\times\{1,\hdots,d\})$:
\begin{align*}
    \EE\big({\textstyle \int_0^T\!\langle\phi,\Lambda\super{n}(t)\rangle_{x,l}\,dt}\big)^2 \to \big({\textstyle \int_0^T\!\langle\phi,\rho^2(t)\rangle_{x,l}\,dt}\big)^2,
\end{align*}
and hence
\begin{align*}
  \EE\big({\textstyle \int_0^T\!\langle\phi,\Lambda\super{n}(t)-\rho^2(t)\rangle_{x,l}\,dt}\big)^2 \to 0.
\end{align*}
\label{cor:4cor time int}
\end{corollary}
Note that this result is consistent with Corollary~\ref{cor:kptsconvintegral} and Jensen's inequality.

\begin{remark} What can we say about convergence of $\Lambda\super{n}$ in other scaling regimes? 
The bound in Corollary~\ref{cor:2cor int} immediately implies vague tightness in $\M(\TT^d\times\{1,\hdots,d\}\times\lbrack0,T\rbrack)$ for any feasible scaling regime. However, the four-point correlation in Corollary~\ref{cor:kptsconvintegral} show that in the regime $1/\epsilon^{d/2}\gge n$, the second moment of $\Lambda\super{n}$ does not vanish or can even blow up, suggesting that $\Lambda\super{n}$ can not have deterministic cluster points. In those scalings, the (random) limit of $\Lambda\super{n}$ remains an open question.

The proof of the four-points correlation in Corollary~\ref{cor:kptsconvintegral} shows that the problem lies in the initial condition. Indeed, the independent Bernouilli distributions puts too much probability in the event where the initial number of occupied nearest neighbours is high. 
\end{remark}

\section{Tightness}\label{sec:tightness}

In this section we prove (or recall) the tightness of all processes $\rho\super{n}$, $W\super{n}$, $C\super{n}$, $\bar W\super{n}$ and $\bar C\super{n}$. In fact, for the processes that converge to a deterministic limit we will show \emph{exponential} tightness. Although this paper focuses on hydrodynamic limits and we do not study large deviations, the exponential tightness is interesting because it shows that many of the variables have \emph{different} large-deviation speeds.

Recall the different scaling regimes from Subsection~\ref{subsec:scaling}; these play an important role in the tightness arguments.

\subsection{Empirical measure and net flux}

For completeness we recall the tightness result of the empirical measure $\rho\super{n}$ and net flux $\bar W\super{n}$ without proof. In fact, classic results show that these variables are also \emph{exponentially} tight. 
\begin{proposition}[{\cite[Sec.~10.4]{KipnisLandim1999}, \cite[Lem.~3.2]
{BDSGJLL2007}}]\phantom{a}
In \emph{any} feasible scaling regime
\begin{equation*}
  n  \lle \frac{1}{\epsilon^d},
\end{equation*}
\begin{itemize}
\item For arbitrary $\phi\in C^\infty(\TT^d)$, the process $t\mapsto\langle\phi,\rho\super{n}(t)\rangle$ is exponentially tight in $D(0,T)$ with speed $n$;
\item For arbitrary $\phi\in C^\infty(\TT^d\times\{1,\hdots,d\})$, the proces $t\mapsto\langle\phi,\bar W\super{n}(t)\rangle$ is exponentially tight in $D(0,T)$ with speed $n$.
\end{itemize}
\end{proposition}

\begin{remark}
Strictly speaking these results are proven in the classic regime only, but it is not difficult to expand these to any feasible scaling regime. 
\end{remark}

\subsection{Unidirectional flux}

Tightness for the unidirectional flux is rather straightforward. We only prove exponential tightness, which directly implies tightness. 

\begin{proposition} In \emph{any} feasible scaling regime
\begin{equation*}
  n  \lle \frac{1}{\epsilon^d},
\end{equation*}
and for arbitrary $\phi\in C^\infty(\TT^d \times \lbrace 1, \ldots, d\rbrace \times \lbrace +, -\rbrace)$, the process $t\mapsto\langle\phi,W\super{n}(t)\rangle_{x,l,\pm}$ is exponentially tight in $D(0,T)$ with speed $n/\epsilon^2$.
\label{prop:uniflux exp tight}
\end{proposition}
\begin{proof} We prove the two statements of Lemma~\ref{lem:tightness criteria}(B).

\begin{enumerate}[(i)]
\item
For the modulus of continuity, we first estimate, for arbitrary $\phi\in C^\infty(\TT^d \times \lbrace 1, \ldots, d\rbrace \times \lbrace +, -\rbrace)$, $\delta>0$:
\begin{align*}
  \omega'(\langle \phi,W\super{n}\rangle_{x,l,\pm},\delta,T)
  &:=\inf_{\substack{0\leq t_0\hdots\leq t_K\leq T \\ t_i-t_{i-1}>\delta}} \max_{i=0,\hdots,K-1} \sup_{s,t\in\lbrack t_i,t_{i+1})} \lvert \langle \phi,W\super{n}(t)-W\super{n}(s)\rangle_{x,l,\pm}\rvert\\
  &\leq \max_{i=0,\hdots,\lfloor T/(2\delta)\rfloor-1} \lVert \phi\rVert_C \lVert W\super{n}(2\delta (i+1))-W\super{n}(2\delta i)\rVert_\TV,
\end{align*}
where we choose a uniform partition\footnote{Strictly speaking the last interval may be a bit bigger than $2\delta$.} of $\lbrack0,T\rbrack$ with lattice spacing $2\delta$, and the supremum over $s,t$ can be replaced by $s=2\delta i, t=2\delta(i+1)$ because each process $W\super{n}_{l,\pm}(t)$ is monotonically increasing. Let $\mathcal{N}(\lambda)$ denote a Poisson variable with intensity $\lambda$. We shall choose $\lambda:=2\delta n\epsilon^{-2}$, where $n$ comes from the total number of particles, $\epsilon^{-2}$ from the time rescaling of each one of them, and $2\delta$ from the time interval. A Chernoff bound yields, for arbitrary $r>0$ and $\gamma>0$:
\begin{align*}
  &\PP\big(\omega'(\langle\phi, W\super{n}\rangle_{x,l,\pm},\delta,T)>r\big)\\
  &\qquad\leq \sum_{i=0}^{\lfloor T/(2\delta)\rfloor-1} \PP\big( \lVert\phi\rVert_C \lVert W\super{n}(2\delta (i+1))-W\super{n}(2\delta i)\rVert_\TV >r\big)\\
  &\qquad\leq \sum_{i=0}^{\lfloor T/(2\delta)\rfloor-1} e^{-n\epsilon^{-2}\gamma r/\lVert\phi\rVert_C}\EE e^{\gamma \mathcal{N}(2\delta n\epsilon^{-2})}\\
  &\qquad=\sum_{i=0}^{\lfloor T/(2\delta)\rfloor-1} e^{-n\epsilon^{-2}\gamma r/\lVert\phi\rVert_C + 2\delta n\epsilon^{-2}(e^\gamma-1)}.
\end{align*}
Therefore, after optimising over $\gamma$, we find in the logarithmic scaling:
\begin{align*}
  &\limsup_{\delta\to0}\limsup_{n\to\infty} \frac1{n\epsilon^{-2}}\log\PP\big(\omega'(\langle \phi,W\super{n}\rangle_{x,l,\pm},\delta,T)>r\big)\\
  &\qquad\leq
  \limsup_{\delta\to0}  \max_{i=0,\hdots \lfloor T/(2\delta)\rfloor-1} -\sup_{\gamma>0} \big\{\mfrac{\gamma r}{\lVert\phi\rVert_C} - 2\delta (e^\gamma-1)\big\}\\
  &\qquad\leq
  \limsup_{\delta\to0}  - s\big(r/\lVert\phi\rVert_C \,\mid\, 2\delta\big) =-\infty, 
\end{align*}
where the relative Boltzmann function is defined as:
\begin{equation}
  s(a\mid b):=
  \begin{cases}
    a\log\frac{a}{b}-a+b, &\text{if } a,b>0,\\
    b,                    &\text{if } a=0.
  \end{cases}
\label{eq:Boltzmann}
\end{equation}

\item For the time marginals, a similar calculation as above shows that,
\begin{multline*}
  \limsup_{n\to\infty} \frac1{n\epsilon^{-2}}\log\PP\big(\lvert\langle\phi,W\super{n}(t)\rangle_{x,l,\pm}\rvert >r \big)\\ 
  \leq \limsup_{n\to\infty} \frac1{n\epsilon^{-2}}\log\PP\big(\lvert\langle\phi,W\super{n}(T)\rangle_{x,l,\pm}\rvert >r \big) \leq - s\big(r/\lVert\phi\rVert_C \,\mid\, T\big).
\end{multline*}
The right-hand side can be made arbitrarily negative by choosing $r$ large enough. This shows the exponential tightness of all time marginals $\langle\phi,W\super{n}(t)\rangle$, $t\in\lbrack0,T\rbrack$.
\end{enumerate}
\end{proof}

\subsection{Unidirectional collision number}

Tightness for the unidirectional collision number $C\super{n}(t)$ is more subtle than for the variables discussed above. In particular, we need to restrict the scaling regime, not only to make sure that the prefactor $\epsilon^2/(\epsilon^d n^2)$ vanishes, but also to guarantee that the terms that appear from overlapping points in Corollary~\ref{cor:4cor int} do not blow up.

We exploit the fact that both $C\super{n}_{\cdot,+}(t)$ and $C\super{n}_{\cdot,-}(t)$ are \emph{Cox processes} in $\lbrack0,T\rbrack\times\TT^d$, that is, Poisson point processes with random intensity:
\begin{equation}
  C\super{n}_{\cdot,\pm} \sim \mfrac{\epsilon^2}{\epsilon^d n^2} \mathrm{PPP}\Big( \mfrac{\epsilon^d n^2}{\epsilon^2} \Lambda\super{n}\Big),
\label{eq:unicol Cox}
\end{equation}
where the nearest-neighbour measure $\Lambda\super{n}(t)$ is defined in \eqref{eq:NN measure}. Indeed, the same measure $\Lambda\super{n}(t)$ causes both positive and negative jumps, because in $\Lambda\super{n}$ each pair $x\sim y$ is only counted once, but in the generator they are counted twice. Moreover, we exploit that after conditioning on $\Lambda\super{n}$ all jumps become independent.

\begin{proposition} Let initial Assumptions~\ref{ass:initial} hold, and assume that 
\begin{align*}
  \frac1{\epsilon^{d/2}} \lle n \lle \frac{1}{\epsilon^d}.
\end{align*}
Then 
for arbitrary $\phi\in C^\infty(\TT^d \times \lbrace 1, \ldots, d\rbrace \times \lbrace +, -\rbrace)$, the process $t\mapsto\langle\phi,C\super{n}(t)\rangle_{x,l,\pm}$ is tight in $D(0,T)$.
\label{prop:unicol tight}
\end{proposition}
\begin{proof} We check the two conditions of Lemma~\ref{lem:tightness criteria}(A). 
\begin{enumerate}[(i)]
\item Take arbitrary $\phi\in C^\infty(\TT^d \times \lbrace 1, \ldots, d\rbrace \times \lbrace +, -\rbrace)$, $\delta>0$ and $r>0$. For brevity the index $i$ will implicitly run from $0$ to $\lfloor T/(2\delta)\rfloor-1$, and we further abbreviate:
\begin{align*}
  A\super{n}_i&:=\lVert C\super{n}(2\delta(i+1))-C\super{n}(2\delta i)\rVert_\TV,\\
  \Lambda\super{n}_i&:=\int_{2\delta i}^{2\delta(i+1)}\!\lVert\Lambda\super{n}(t)\rVert_\TV\,dt.
\end{align*}

After conditioning we use the above-mentioned independence, the Weierstrass product inequality and finally the Chebychev inequality to estimate:
\begin{align*}
  &\PP\big(\omega'(\langle\phi,C\super{n}\rangle_{x,l,\pm},\delta,T)>r\big)\\
  &\quad \leq \PP\big( \max_i A\super{n}_i >r/ \lVert \phi\rVert_C\big)\\
  &\quad = 1 - \int\! \PP\big( \forall i\, A\super{n}_i \leq r/ \lVert \phi\rVert_C \,\mid\,(\Lambda\super{n}_i)_i=(\lambda_i)_i\big)\,\PP\big( (\Lambda\super{n}_i)_i=(\lambda_i)_i\big) \\
  &\quad= 1 - \int\!\prod_i\Big( 1-\PP( A\super{n}_i >r/ \lVert \phi\rVert_C \,\mid\,\Lambda\super{n}_i=\lambda_i\big)\Big)\,\PP\big( (\Lambda\super{n}_i)_i=(\lambda_i)_i\big) \\
  &\quad\leq \sum_i \int\!\PP( A\super{n}_i >r/ \lVert \phi\rVert_C \,\mid\,\Lambda\super{n}_i=\lambda_i\big)\Big)\,\PP\big( (\Lambda\super{n}_i)_i=(\lambda_i)_i\big) \\  
  &\quad\leq \sum_i \int\!\frac{\EE\big\lbrack (A\super{n}_i)^2 \,\mid\,\Lambda\super{n}_i=\lambda_i \big\rbrack}{r^2/\lVert\phi\rVert_C^2} \,\PP\big( (\Lambda\super{n}_i)_i=(\lambda_i)_i\big) \\ 
  &\quad= \sum_i \int\! \frac{\epsilon^4}{\epsilon^{2d}n^4}  \cdot \frac{(\tfrac{2\epsilon^dn^2}{\epsilon^2}\lambda_i)^2 + (\tfrac{2\epsilon^dn^2}{\epsilon^2}\lambda_i)}{r^2/\lVert\phi\rVert_C^2} \,\PP\big( (\Lambda\super{n}_i)_i=(\lambda_i)_i\big)\\   
  &\quad= \frac{4}{r^2/\lVert\phi\rVert_C^2}\sum_i\EE (\Lambda\super{n}_i)^2 + \frac{2}{r^2/\lVert\phi\rVert_C^2}\frac{\epsilon^2}{\epsilon^d n^2} \sum_i\EE\Lambda\super{n}_i,
\end{align*}
since the conditioned variable $A_i$ is simply $\epsilon^2/(\epsilon^d n^2)$ times a Poisson distribution with intensity $2\epsilon^dn^2/\epsilon^2\lambda_i$. The factor $2$ appears here because each neigbouring pair in $\Lambda\super{n}$ drives the jumps in positive and negative directions.
\\
From here we use Corollaries~\ref{cor:2cor int} and \ref{cor:4cor int}, and the assumption on the scaling regime to obtain $\kappa_2,\kappa_4>0$ such that
\begin{align*}
  \EE\Lambda\super{n}_i \leq \delta \kappa_2,
  &&
  \EE\big(\Lambda\super{n}_i\big)^2 \leq \delta^2 \kappa_4.
\end{align*}
Hence we can further estimate::
\begin{align*}
  &\PP\big(\omega'(\langle\phi,C\super{n}\rangle_{x,l,\pm},\delta,T)>r\big)\\
  &\quad\leq\frac{4\delta^2\kappa_4}{r^2/\lVert\phi\rVert_C^2} \lfloor \tfrac{T}{2\delta}\rfloor + \frac{\epsilon^2}{\epsilon^d n^2} \frac{2\delta \kappa_2}{r^2/\lVert\phi\rVert_C^2} \lfloor \tfrac{T}{2\delta}\rfloor \\
  &\quad\xrightarrow{n\to\infty} \frac{4\delta^2\kappa_4}{r^2/\lVert\phi\rVert_C^2} \lfloor \tfrac{T}{2\delta}\rfloor
  \xrightarrow{\delta\to0} 0,
\end{align*}
where we explicitly used the assumption on the scaling regime.

\item For the second estimate, we can use a similar argument as above, but now with a straightforward Markov inequality. For arbitrary $\phi\in C(\TT^d \times \lbrace 1, \ldots, d\rbrace \times \lbrace +, -\rbrace)$:
\begin{align*}
  &\limsup_{r\to\infty}\limsup_{n\to\infty}\PP\big( \sup_{t\in\lbrack0,T\rbrack} \lvert \langle\phi,C\super{n}(t)\rangle_{x,l,\pm}\rvert>r\big)\\
  &\qquad\leq
  \limsup_{r\to\infty}\limsup_{n\to\infty} \PP\big(\lVert C\super{n}(T)\rVert_\TV > r/\lVert \phi\rVert_C \big)\\
  &\qquad\leq
  \limsup_{r\to\infty}\limsup_{n\to\infty}  \frac{1}{r/\lVert\phi\rVert_C} \frac{\epsilon^2}{\epsilon^d n^2} \EE\Big\lbrack{\textstyle 2\frac{\epsilon^dn^2}{\epsilon^2}\int_0^T\!\lVert\Lambda\super{n}(t)\rVert_\TV\,dt }\Big\rbrack\\
  &\hspace{0.65em} \stackrel{\text{(Cor.~\ref{cor:2cor int})}}{\leq} \limsup_{r\to\infty}\frac{1}{r/\lVert\phi\rVert_C} 2d  T\lVert\rho_0\rVert_C^2 = 0.  
\end{align*}
which was to be shown.
\end{enumerate}
\end{proof}

\begin{remark}
The restriction to $\epsilon^{-d/2}\lle n$ is more related to the initial condition than the dynamics, as can be seen from the proof of Corollary~\ref{cor:kptsconvintegral}. Therefore, we believe that this restriction can be relaxed by considering more elaborate initial conditions, where occupied nearest neighbours are penalised. This is beyond the scope of this paper.
\label{rem:unicol regimes}
\end{remark}

We now briefly study the exponential tightness of the unidirectional collision number. One difficulty here is that on the large-deviation scale of interest, the total number of particles could be so large that useful bounds on $\lVert\Lambda\super{n}\rVert_\TV$ are lacking. We only provide a straight-forward argument for the classic regime, and conditioned on the total number of particles, which is already interesting because the speed is different from $n$. A deeper study into the exponential tightness and large deviations are beyond the scope of this paper.

\begin{theorem} In the classic scaling regime
\begin{align*}
  n \sim \frac1{\epsilon^d}, &&\text{i.e. } n=\frac{\alpha}{\epsilon^d}, \alpha\in(0,1),
\end{align*}
and conditioned on $\lVert\Rho\super{n}(0)\rVert_\TV=\lVert\rho_0\rVert_{L^1(\TT^d)}$, for arbitrary $\phi\in C^\infty(\TT^d \times \lbrace 1, \ldots, d\rbrace \times \lbrace +, -\rbrace)$, the process $t\mapsto\langle\phi,W\super{n}(t)\rangle$ is exponentially tight in $D(0,T)$ with speed:
\begin{align*}
  \frac{\epsilon^d n^2}{\epsilon^2}=\frac{\alpha^2}{\epsilon^{d+2}}=\alpha^{1-2/d} n^{1+2/d}.
\end{align*}
\end{theorem}
\begin{proof} The conditioning says that the total number of particles in the system is almost surely $n\lVert\rho_0\rVert_{L^1(\TT^d)}$, and so
\begin{align*}
  \sup_{t\in\lbrack0,T\rbrack} \lVert\Lambda\super{n}(t)\rVert_\TV
  \leq 
  \frac{d\lVert\rho_0\rVert_{L^1(\TT^d)}}{\epsilon^d n}
  =
  \frac{d\lVert\rho_0\rVert_{L^1(\TT^d)}}{\alpha},
  &&
  \epsilon^d n = \alpha\in(0,1).
\end{align*}
We now exploit this bound together with
\begin{align*}
  \lVert C\super{n}(t)\rVert_\TV 
  \sim
  2\mfrac{\epsilon^2}{\epsilon^d n^2} \mathrm{Poi}\Big( \mfrac{\epsilon^d n^2}{\epsilon^2} \lVert\Lambda\super{n}(t)\rVert_\TV\Big).
\end{align*}
The result then follows from the same argument as in the proof of Proposition~\ref{prop:uniflux exp tight}.
\end{proof}

\begin{remark} It is tempting to do a formal large-deviation calculation with the non-linear semigroup~\cite[Sec. A1.7]{KipnisLandim1999}. However, in such calculation one needs to replace $\Lambda\super{n}$ by $\rho^2(x)$ on the large-deviation scale, see \cite[Sec.~10.3]{KipnisLandim1999}. Such replacement lemma/superexponential estimate is not known -- and we suspect even wrong -- in any other scaling regime than the classic one.
\end{remark}

\subsection{Net collision number}\label{sec. tightness net coll}

We finally show tightness of the net flux; our only variable that will converge to a non-deterministic limit.
\begin{proposition} 
Assume that
\begin{equation*}
  \frac{1}{\epsilon^{d/2}} \lle n \lle \frac{1}{\epsilon^d}.
\end{equation*}
Then for arbitrary $\phi\in C^\infty(\TT^d\times\{1,\hdots,d\})$, the process $t\mapsto\langle\phi,\bar C\super{n}(t)\rangle_{x,l}$ is tight in $D(0,T)$.
\label{prop:netflux tight}
\end{proposition}

\begin{proof} We now check the conditions of Lemma~\ref{lem:tightness criteria}\eqref{lem:tightness criteria ferrari} for an arbitrary $\phi\in C^\infty(\TT^d\times\{1,\hdots,d\})$.

\begin{enumerate}
\item[\eqref{it:ferrari bdd carre}] An explicit calculation shows:
\begin{align*}
  (\Gamma\super{n}_2\langle \phi,\cdot\rangle_{x,l})(\bar C\super{n}(t))
  &= 2 \langle\phi^2,\Lambda\super{n}(t)\rangle_{x,l}, 
\end{align*}
 and so by Lemma~\ref{lem:4cor estimate},
\begin{align*}
  \EE \left[(\Gamma\super{n}_2\langle \phi,\cdot\rangle_{x,l})(\bar C\super{n}(t)) \right]^2
  &\leq
4\lVert\phi\rVert_C^2 \Big(
    d^2\lVert\rho_0\rVert_C^4
    +
    \mfrac{4d}{n}  \lVert \rho_0\rVert_C^3
    +
    \mfrac{d}{\epsilon^d n^2} \lVert \rho_0\rVert_C^2 \Big),
\end{align*}
which is bounded by the assumption on the scaling.

\item[\eqref{it:ferrari bdd generator}]
For this particular process $\Q\super{n}\langle\phi,\cdot\rangle_{x,l} \equiv 0$, so clearly the squared generator is uniformly bounded.

\item[\eqref{it:ferrari bdd square}] For the same reason $\Q\super{n}\langle\phi,\cdot\rangle_{x,l}^2=\Gamma\super{n}_2\langle\phi,\cdot\rangle_{x,l}$, which implies
\begin{align*}
  \frac{d}{dt} \EE \left[\langle \phi,\bar C\super{n}(t)\rangle_{x,l}^2 \right]
  &=\EE \left[(\Gamma\super{n}_2\langle \phi,\cdot\rangle_{x,l})(\bar C\super{n}(t) \right]\\
  &= 2 \EE\langle\phi^2,\Lambda\super{n}(t)\rangle_{x,l}  \leq 2 d  \lVert\phi^2\rVert_C \lVert\rho_0\rVert_C^2,
\end{align*}
where we applied Lemma~\ref{lem:2cor estimate}. The claim follows by Gr{\"o}nwell.

\item[\eqref{it:ferrari bdd jumps}]
For arbitrary $t\in\lbrack0,T\rbrack$ the jump size is deterministically bounded:
\begin{align*}
  \big\lvert \langle \phi,\bar{C}\super{n}(t)-\bar{C}\super{n}(t^-)\rangle_{x,l} \big\rvert \leq \frac{\epsilon}{\sqrt{\epsilon^d}n}\lVert\phi\rVert_{L^\infty}=:2\delta_n,
\end{align*}
so that $\delta_n\to0$ by the assumption on the scaling. Then clearly 
\[
\PP\Big({\textstyle \sup_{ 0 \leq t \leq T} \big\lvert \langle \phi,\bar{C}\super{n}(t)-\bar{C}\super{n}(t^-)\rangle_{x,l} \big\rvert \geq \delta_n} \Big) = 0
\]
as was to be shown.
\end{enumerate}
\end{proof}

\begin{remark}
For this tightness result we need two nontrivial restrictions on the scaling regime:
\begin{itemize}
\item In order to bound the expected squared carré du champ in (iii), we need $\epsilon^{-d/2}\lle n$,
\item In order for $\delta_n\to0$ in condition (iv) we need $\epsilon^{1-d/2}\lle n$.
\end{itemize}
As in Remark~\ref{rem:unicol regimes}, we expect that the first restriction can be relaxed by considering a more elaborate initial condition. However, one cannot go beyond the second restriction, since this says that the prefactor $\epsilon/(\sqrt{\epsilon^d}n)$ in the definition of $\bar{C}\super{n}$ (see Subsection~\ref{subsec:dynamics}) should vanish in order to have a law of large numbers.
\label{rem:netcol regimes}
\end{remark}


\section{Main result: hydrodynamic limits}\label{sec:mainresults}

We are now ready to piece together the hydrodynamic limits of the random variables $W\super{n}(t),C\super{n}(t),\bar W\super{n}(t),\bar C\super{n}(t)$ introduced in Subsection~\ref{subsec:dynamics}. These define a mapping $\nu_n:(\eta,\widetilde W\super{n},\widetilde C\super{n})\mapsto (\eta,W\super{n},C\super{n},\bar{W}\super{n},\bar{C}\super{n})$ from the unscaled to the rescaled variables. Since the jump rates in the generator $\tilde\Q$ from~\eqref{eq:raw generator} depend on $\eta$ only, the process $\big(\eta(t),W\super{n}(t),C\super{n}(t),\bar{W}\super{n}(t),\bar{C}\super{n}(t)\big)$ is again Markovian, with generator (for any choice $(\eta,\widetilde w,\widetilde c)\in\nu_n^{-1}\lbrack(\eta,w,c,\bar{w},\bar{c})\rbrack$):
\begin{align}
  (\Q\super{n} \Phi)(\eta,w,c,\bar{w},\bar{c})=(\widetilde{\Q}\super{n} (\Phi\circ\nu_n))(\eta,\widetilde w,\widetilde c),
\label{eq:generator}
\end{align}
We use this generator to show the limits of each of the variables separately in the coming subsections, where we usually suppress the dependency of the generators on $\eta$.

Recall Assumption~\ref{ass:initial} on the initial configuration for a given continuous density $\rho_0$, and the corresponding time-evolved densities $\rho(x,t)$ from~\eqref{eq:rho}.

\subsection{Empirical measure and net flux}

As mentioned in the introduction, the hydrodynamic limit for the empirical measure and net flux are classic results, and we only state them here for completeness. For proofs we refer to the references from the introduction. Possibly the only new result is that the hydrodynamic limit is independent of the scaling $\epsilon^d n$, which can easily be checked using the techniques from this paper.

\begin{theorem} Under Assumption~\ref{ass:initial} on the initial configuration, and in any feasible scaling regime
\begin{equation*}
    n \lle \frac1{\epsilon^d},
\end{equation*}
the empirical measure $\rho\super{n}$ and net flux $\bar{W}\super{n}$ converge narrowly in the Skorohod-Mitoma topology to the (unique) deterministic path $\rho,\bar{w}$ that satisfy:
\begin{align*}
  \begin{cases}
    \partial_t \rho(x,t)=-\div\big (\partial_t\bar{w}(x,t)\big), & 
    \partial_t \bar{w}(x,t)=-\grad \rho(x,t),\\
    \rho(x,0)=\rho_0(x), & 
    \bar{w}_{l}(x,0)\equiv 0.
  \end{cases}\\
\end{align*}
\label{th:hydrolim netflux}
\end{theorem}

\subsection{Unidirectional flux}

\begin{theorem} Under Assumption~\ref{ass:initial} on the initial configuration, and in any feasible scaling regime
\begin{align*}
    n \lle \frac{1}{\epsilon^d}, && \alpha:=\lim_{n\to\infty} \epsilon^d n \lVert\rho_0\rVert_{L^1}\in\lbrack0,1),
\end{align*}
the unidirectional flux $W\super{n}$ converges narrowly in the Skorohod-Mitoma topology to the (unique) deterministic path $w$ that satisfies the equation:
\begin{align*}
  \begin{cases}
    \partial_t w_{l,\pm}(x,t)=\rho(x,t)-\frac{\alpha}{\lVert\rho_0\rVert_{L^1}}\rho^2(x,t)=\begin{cases}\rho(x,t),&\frac{1}{\epsilon^d}=n,\\ \rho(x,t)-\frac{\alpha}{\lVert\rho_0\rVert_{L^1}}\rho^2(x,t), & \frac{1}{\epsilon^d}\ll n, \end{cases}\\
    w_{l,\pm}(x,0)\equiv 0.
  \end{cases}
\end{align*}  
\label{th:hydrolim uniflux}
\end{theorem}
\noindent Note that the limit unidirectional flux is independent of the direction~$l,\pm$.
\begin{proof}
First calculate the generator and carré du champ on linear functionals, for $\phi\in C(\TT^d\times\{1,\hdots,d\}\times\{+,-\})$:
\begin{align}
    (\Q\super{n}\langle\phi,\cdot\rangle_{x,l,\pm})(w) &= \frac1n \sum_{l=1}^d \sum_{\pm=+,-} \sumx\eta(x,t)(1-\eta(x\pm\epsilon\mathds1_l,t))\phi_{l,\pm}(x\pm\tfrac{\epsilon}{2}\mathds1_l),\notag\\
    &=\langle\phi,\Rho\super{n}(t)\rangle_{x,l,\pm}-\epsilon^d n \langle \phi,\Lambda\super{n}(t)\rangle_{x,l,\pm}\label{eq:uniflux Q}\\
    (\Gamma\super{n}_2\langle\phi,\cdot\rangle_{x,l,\pm})(w) &= \frac{\epsilon^2}{n^2} \sum_{l=1}^d \sum_{\pm=+,-} \sumx \eta(x)\big(1-\eta(x\pm\epsilon\mathds1_l)\big) \phi_{l,\pm}^2(x\pm\tfrac{\epsilon}{2}\mathds1_l)\notag\\
    &=\frac{\epsilon^2}{n}\langle\phi^2,\Rho\super{n}(t)\rangle_{x,l,\pm}-\epsilon^{d+2} \langle \phi^2,\Lambda\super{n}(t)\rangle_{x,l,\pm}.
    \label{eq:uniflux Gamma}
\end{align}
By Corollary~\ref{cor:kptsconvintegral} the mean drift converges:
\begin{align*}
  \EE(\Q\super{n}\langle\phi,\cdot\rangle_{x,l,\pm})(w) \to \langle\phi,b(w)\rangle_{x,l,\pm},
  &&
  b_{l,\pm}(w)(x,t):=\rho(x,t) - \frac{\alpha}{\lVert\rho_0\rVert} \rho^2(x,t).
\end{align*}

We need to check the conditions of Theorem~\ref{th:hydrodyn limit det} with this drift $b(w)$. By Mitoma's Theorem~\ref{th:mitoma}, our Proposition~\ref{prop:uniflux exp tight} shows that the process $W\super{n}$ is tight. For Conditions (i) and (ii):
\begin{enumerate}[(i)]
\item By Corollary~\ref{cor:kptsconvintegral}:
\begin{align*}
  &\EE\Big\lbrack (\Q\super{n}\langle\phi,\cdot\rangle_{x,l,\pm})(w)\Big\rbrack^2 \\
  &\stackrel{\eqref{eq:uniflux Q}}{=}
  \EE\langle\phi,\Rho\super{n}(t)\rangle^2_{x,l,\pm}-2\epsilon^d n \EE \langle\phi,\Rho\super{n}(t)\rangle_{x,l,\pm} \langle \phi,\Lambda\super{n}(t)\rangle_{x,l,\pm} \\
  &\hspace{18em} + (\epsilon^d n)^2\EE\langle \phi,\Lambda\super{n}(t)\rangle^2_{x,l,\pm}\\
  &\to
  \langle\phi,b(w)\rangle^2_{x,l,\pm}.
\end{align*}
\item For the quadratic variation, again using Corollary~\ref{cor:kptsconvintegral}:
\begin{align*}
  \EE\langle M\super{n,\phi}\rangle_T &=\EE\int_0^T\!(\Gamma\super{n}_2\langle\phi,\cdot\rangle_{x,l,\pm})(W\super{n}(t))\,dt\\
  &\!\!\stackrel{\eqref{eq:uniflux Gamma}}{=}
    \frac{\epsilon^2}{n}\EE\int_0^T\!\langle\phi^2,\Rho\super{n}(t)\rangle_{x,l,\pm}\,dt-\epsilon^{d+2} \EE\int_0^T\!\langle \phi^2,\Lambda\super{n}(t)\rangle_{x,l,\pm}\,dt 
  \to 0.
\end{align*}
\end{enumerate}

\end{proof}

\subsection{Unidirectional collision number}

\begin{theorem} Under Assumption~\ref{ass:initial} on the initial configuration, and in the scaling regime
\begin{align*}
  \frac1{\epsilon^{d/2}} \ll n \lle \frac{1}{\epsilon^d},
\end{align*}
the unidirectional collision number $C\super{n}$ converges narrowly in the Skorohod-Mitoma topology to the (unique) deterministic path $c$ that satisfies the equation:
\begin{align*}
  \begin{cases}
    \partial_t c_{l,\pm}(x,t)=\rho^2(x,t),\\
    c_{l,\pm}(x,0)\equiv 0.
  \end{cases}
\end{align*}
\label{th:hydrolim unicol}
\end{theorem}
\begin{proof} The proof is similar to Theorem~\ref{th:hydrolim uniflux} for the unidirectional flux, except for the quadration variation. First calculate the generator and carré du champ on linear functionals $\langle\phi,c\rangle_{x,l,\pm}$
\begin{align}
    (\Q\super{n}\langle\phi,\cdot\rangle_{x,l,\pm})(w) &= \langle\phi,\Lambda\super{n}(t)\rangle_{x,l,\pm}
    \label{eq:unicol Q}\\
    (\Gamma\super{n}_2\langle\phi,\cdot\rangle_{x,l,\pm})(w) &= \frac{\epsilon^2}{\epsilon^{d} n^2} \langle\phi^2,\Lambda\super{n}(t)\rangle_{x,l,\pm}
    \label{eq:unicol Gamma}
\end{align}
By Corollary~\ref{cor:kptsconvintegral} the mean drift converges:
\begin{align*}
  \EE(\Q\super{n}\langle\phi,\cdot\rangle_{x,l,\pm})(w) \to \langle\phi,b(w)\rangle_{x,l,\pm},
  &&
  b_{l,\pm}(w)(x,t):=\rho^2(x,t).
\end{align*}
We check the conditions of Theorem 3.1 with this drift b(w), where tightness for $C\super{n}$ holds by Mitoma’s Theorem 3.6 in combination with our Proposition~\ref{prop:unicol tight}. To check Conditions (i) and (ii):
\begin{enumerate}[(i)]
\item By Corollary~\ref{cor:kptsconvintegral}:
\begin{align}
&\EE\Big\lbrack (\Q\super{n}\langle\phi,\cdot\rangle_{x,l,\pm})(w)\Big\rbrack^2 \stackrel{\eqref{eq:unicol Q}}{=} \EE\Big\lbrack \langle\phi,\Lambda\super{n}(t)\rangle_{x,l,\pm} \Big\rbrack^2 
\to\langle\phi,b(w)\rangle_{x,l,\pm}^2.
\label{eq:unicol condition i}
\end{align}
\item For this quadratic variation, the time integral requires the more subtle bound from Lemma~\ref{lem:2cor estimate}:
\begin{align*}
  &\EE\langle M\super{n,\phi}\rangle_T = \EE\int_0^T\!(\Gamma\super{n}_2\langle\phi,\cdot\rangle_{x,l,\pm})(C\super{n}(t))\,dt \stackrel{\eqref{eq:unicol Gamma}}{=}
    \frac{\epsilon^2}{\epsilon^{d} n^2} \EE\int_0^T\!  \langle\phi^2,\Lambda\super{n}(t)\rangle_{x,l,\pm}\,dt\\
  &\qquad \leq \frac{\epsilon^2}{\epsilon^d n^2} 2d\lVert\phi\rVert_C\,\lVert\rho_0\rVert^2_C\to 0.
  \hspace{3em} \text{(by the assumption on the scaling)}
\end{align*}
\end{enumerate}

\end{proof}

\begin{remark} Interestingly, the unidirectional collision number $C\super{n}$ is tight up until $\epsilon^{d/2}=n$, but at that critical regime $\Lambda\super{n}$ no longer has a deterministic limit, and \eqref{eq:unicol condition i} would fail due to Corollary~\ref{cor:kptsconvintegral}. In other words, $C\super{n}$ should have a random limit, but we cannot identify that limit with our current techniques.
\end{remark}

\subsection{Net collision number}

\begin{theorem} Under Assumption~\ref{ass:initial} on the initial configuration, and in any feasible scaling regime
\begin{equation*}
    \frac{1}{\epsilon^{d/2}}\ll n \lle \frac{1}{\epsilon^d},
\end{equation*}
the net collision number $\bar{C}\super{n}$ converges narrowly in the Skorohod-Mitoma topology to the (unique) random path $\bar{C}$ that weakly satisfies the SPDE:
\begin{align*}
  \begin{cases}
    \partial_t \bar{C}_{l}(dx,dt)=2\rho(x,t)\Xi(dx,dt),\\
    \bar{C}_{l}(x,0)\equiv 0,
  \end{cases}
\end{align*}
with space-time white noise $\Xi(dx,dt)$.
\label{th:hydrolim netcol}
\end{theorem}
\begin{proof}
Again, we start by calculating the generator and carré du champ on linear functionals $\langle\phi,\bar c\rangle_{x,l}$. From~\eqref{eq. k-field power} and the computations of Section~\ref{sec. tightness net coll} we have:
\begin{align}
    (\Q\super{n}\langle\phi,\cdot\rangle_{x,l})( \bar{C}\super{n}(t)) &=0,
    \label{eq:netcol Q}\\
    (\Gamma_2\super{n}\langle\phi,\cdot\rangle_{x,l})(\bar{C}\super{n}(t))
    &= 2 \langle\phi^2,\Lambda\super{n}(t)\rangle_{x,l},
    \label{eq:netcol Gamma_2}\\
    (\Gamma_3\super{n}\langle\phi,\cdot\rangle_{x,l})(\bar{C}\super{n}(t)) &=0,\label{eq:netcol Gamma_3}\\
    (\Gamma_4\super{n}\langle\phi,\cdot\rangle_{x,l})(\bar{C}\super{n}(t)) &= \frac{2\epsilon^2}{\epsilon^d n^2} \langle\phi^4,\Lambda\super{n}(t)\rangle_{x,l}. 
    \label{eq:netcol Gamma_4}
\end{align}
In order for these to stay bounded, we will need that the prefactor in $\Gamma\super{4}_2$ does not blow up, which is certainly the case since by assumption (see Remark~\ref{rem:netcol regimes}):
\begin{align}
  \epsilon^{1-d/2}\ll \epsilon^{-d/2} \lle n.
\label{eq:netcol scaling regime}
\end{align}

From the calculations above, we determine the candidates for the drift $b(\bar c)$ and covariation $\sigma(\bar c)$, see Remark~\ref{rem:b and sigma}. Clearly the drift is trivial $b(\bar c)\equiv0$, and as for the covariation, Corollary~\ref{cor:2cor int} implies that:

\begin{align}
  \EE\langle M\super{n,\phi}\rangle_t &= \int_0^t\!\EE \, \Gamma\super{n}_2\langle\phi,\cdot\rangle_{x,l}\big(\bar{C}\super{n}(s)\big)\,ds\notag\\
  &\!\!\stackrel{\eqref{eq:netcol Gamma_2}}{=}
  2 \int_0^t\!\EE
      \langle\phi^2,\Lambda\super{n}(s)\rangle_{x,l}\,ds  \to 2  \int_0^t\! \langle\phi^2(s),\rho^2(s)\rangle_{x,l}\, ds.
\label{eq:netcol QV convergence}
\end{align}
We thus need to check the conditions of the stochastic hydrodynamic limit Theorem~\ref{th:hydrodyn limit} with $b(\bar c)\equiv 0$ and $\sigma_l(\bar c)=\sqrt{2}\rho(x)$.

\begin{enumerate}[(i)]
\item Trivial.
\item By Jensen's inequality we only need to show that 
$\int_0^T\! \EE\big\lbrack \Gamma\super{n}_2\langle\phi,\bar C\super{n}(t)\rangle_{x,l}-\int_{\TT^d}\!\phi(x)^2\sigma(x,t)^2\,dx \big\rbrack^2\,dt\to0$, and convergence of the cross product is already shown in \eqref{eq:netcol QV convergence} above. What remains to show is convergence of the squares:
\begin{align*}
  \int_0^T\!\EE\big\lbrack \Gamma\super{n}_2\langle\phi,\bar C(t)\rangle_{x,l}\big\rbrack^2\,dt &= 4 \int_0^T\!\EE 
    \langle\phi^2(t),\Lambda\super{n}(t)\rangle_{x,l}^2  \,dt,
\end{align*}
which by Corollary~\ref{cor:4cor time int}, in the scaling regime $\epsilon^{-d/2}\ll n \lle \epsilon^{-d}$, converges to
\begin{align*}
 \int_0^T\! \langle\phi^2(t),\rho^2(t)\rangle_{x,l}^2 \,dt.
\end{align*}
\item Let $c_n:=\frac{\epsilon}{\sqrt{\epsilon^d}n} \lVert\phi\rVert_{L^\infty}$. By~\eqref{eq:netcol scaling regime}, we have $c_n \to 0$. Moreover, for arbitrary $t\in\lbrack0,T\rbrack$ the jump size is deterministically bounded:
\begin{align*}
  \big\lvert \langle \phi,\bar{C}\super{n}(t)-\bar{C}\super{n}(t^-)\rangle_{x,l} \big\rvert \leq c_n.
\end{align*}
Then clearly 
\[
\lim_{n \to \infty} \PP\Big\lbrack{ \textstyle \sup_{t \in [0,T]} \lvert \langle \phi,\bar{C}\super{n}(t)\rangle -\langle \phi,\bar{C}\super{n}(t^-)\rangle_{x,l} \rvert > c_n } \Big\rbrack = 0,
\]
as was to be shown.
\item From \eqref{eq:netcol Q} and \eqref{eq:netcol Gamma_3}, it is enough to consider $\Gamma_k\super{n}$ for $k \in \lbrace 2, 4\rbrace$. For such $k$, using \eqref{eq:netcol Gamma_2} and \eqref{eq:netcol Gamma_4}, we have
\begin{align*}
 \sup_{n \in \NN} \sup_{ 0 \leq t \leq T} \EE \left[ (\Gamma_k\super{n}\langle\phi,\cdot\rangle_{x,l})(\bar{C}\super{n}(t))\right]^2 \leq \sup_{n \in \NN} \sup_{ 0 \leq t \leq T} 4 \EE \langle\phi^k,\Lambda\super{n}(t)\rangle_{x,l}^2 ,
\end{align*}
where we used \eqref{eq:netcol scaling regime} to bound $\frac{\epsilon^2}{\epsilon^d n^2}<1$. Using Lemma~\ref{lem:4cor estimate} we conclude
\begin{align*}
 &\sup_{n \in \NN} \sup_{ 0 \leq t \leq T} \EE \left[ (\Gamma_k\super{n}\langle\phi,\cdot\rangle_{x,l})(\bar{C}\super{n}(t))\right]^2 \nonumber \\
 &\leq \sup_{n \in \NN} \sup_{ 0 \leq t \leq T} 4 \lVert\phi^k\rVert_C^2  \Big(
    d^2\lVert\rho_0\rVert_C^4
    +
    \mfrac{4d}{n}  \lVert \rho_0\rVert_C^3
    +
    \mfrac{d}{\epsilon^d n^2} \lVert \rho_0\rVert_C^2 \Big) < \infty. 
\end{align*}
\end{enumerate}
\end{proof}

\section{Summary and Discussion}\label{sec: disc}

\renewcommand{\arraystretch}{1.5} 
\begin{tabular}{|l|l|l|l|}
variable & tightness regime & hydrodynamic limit & LDP speed\\
\hline\hline
$\Rho\super{n}$   &$n\lle \frac{1}{\epsilon^d}$ 
  & $\partial_t\rho(t)=\Delta\rho(t)$ & $n$ \\
\hline
$\bar W\super{n}$ &$n\lle \frac{1}{\epsilon^d}$
  & $\partial_t{\bar{w}}(t)=-\nabla\rho(t)$ & $n$\\
\hline
\multirow{2}{*}{$W\super{n}$}
      &$n \sim \frac{1}{\epsilon^d}$ & $\partial_t w(t)=\rho(t)-\alpha/\lVert\rho_0\rVert_{L^1} \rho^2(t)$     &  $n/\epsilon^2$\\
      &$n \ll \frac{1}{\epsilon^d}$ &  $\partial_t w(t)=\rho(t)$    &  $n/\epsilon^2$\\
\hline
\multirow{2}{*}{$\Lambda\super{n}$}
    & $  \frac1{\epsilon^{d/2}} \ll n \lle \frac{1}{\epsilon^d}$ &  $\rho^2$ & ?\\
    & $  \frac1{\epsilon^{d/2}} \gge n $ &  ? (random) & NA\\
\hline
\multirow{2}{*}{$C\super{n}$}
    & $  \frac1{\epsilon^{d/2}} \ll n \lle \frac{1}{\epsilon^d}$ &  $\partial_t c(t)=\rho^2(t)$ & $n^{1+d/2}$?\\
    & $  \frac1{\epsilon^{d/2}} \sim n $ &  ? & ?\\
\hline
\multirow{2}{*}{$\bar{C}\super{n}$}
    &$\frac{1}{\epsilon^{d/2}}\ll n \lle \frac{1}{\epsilon^d}$ & 
$\partial_t{\bar{C}}(dt)=2\rho(t)\Xi(dt)$ & NA\\
    &$\frac{1}{\epsilon^{d/2}}\sim n$ & ? & NA \\
\hline
\end{tabular}

\paragraph{Hydrodynamics.} The hydrodynamic limits for the empirical measure and net flux are classic results. The hydrodynamic limit of the unidirectional flux $w(t)$ depends on the scaling regime; in the classic regime $n\sim1/\epsilon^d$ it sees the effect of the exclusion mechanism, whereas in the sparse regime $n\ll1/\epsilon^d$, it behaves as if the particles were independent. We believe that this effect does not occur for the net flux due to symmetry reasons.

To obtain the hydrodynamic limits for the collision numbers, it is essential to understand the limit of the nearest-neighbour measure $\Lambda\super{n}(t)$, since this determines the jump rates. The fact that the limit unidirectional collision number $c(t)$ grows quadratically in $\rho(t)$ is consistent with the chemical `mass-action' kinetics: each collision requires two particles to be approximately in the same position.

\paragraph{(Thermo)dynamics.} We now revisit the two perspectives from Section~\ref{sec:intro}.
\begin{enumerate}[A.]
\item By rescaling \eqref{eq:unidirectional difference} and passing to the limit we obtain:
\begin{align*}
  \partial_t w(t)
  &=\partial_t w^\IRW(t) - \hspace{1em} \epsilon^d n \,\,\, \partial_t c(t)\\
  &=\rho(t)  \hspace{2.5em}- 
    \begin{cases} 
      \frac{\alpha}{\lVert\rho_0\rVert_C} \rho^2(t), & n \sim \frac1{\epsilon^d},\\
      0, & n \ll \frac1{\epsilon^d}.
    \end{cases}
\end{align*}
This equation shows very nicely that the effect of the exclusion mechanism, namely the collisions, on the unidirectional flux is of order $\epsilon^d n$, and hence this effect vanishes in the sparse regime $n\ll1/\epsilon^d$. Moreover, the total number of collisions decreasing the free energy can now easily be measured by $\epsilon^d n\int_0^T\lVert\rho(t)\rVert^2_{L^2}\,dt$.

\item 
If we do the same thing with the net quantities, we obtain from~\eqref{eq:net difference}:
\end{enumerate}
\begin{align*}
  \partial_t \bar w(t) &= \partial_t \bar w^\IRW(t) \hspace{2.3em}- \sqrt{\epsilon^d}\, \partial_t \bar C(t)\\
  &=-\grad\rho(t) +\text{h.o.t.}  - \sqrt{\epsilon^d} 2\rho(t)\Xi(t) 
\end{align*}
Again, similar to the unidirectional setting, the exclusion mechanism only affects the net flux on a higher-order scale $\sqrt{\epsilon^d}$. It is tempting to apply the continuity equation $\partial_t \mu(t):=-\div(\partial_t \bar{w}(t))$ to derive some type of fluctuating-hydrodynamics equation. However, this naive calculation would ignore another, well-known noise term $2\sqrt{1/n}\div(\sqrt{\mu(t)}\Xi(t))$, here included in the higher-order terms, that already occurs in the \IRW-model~\cite{Dean1996}.  Indeed, the full fluctuating-hydrodynamics equation for the \SEP-model is \cite{ferrari1988non,ravishankar1992fluctuations}:
\begin{align*}
  \partial_t \mu(t) = \Delta \mu(t) + 2 \div\Big(\sqrt{\frac{1}{n}\mu(t)-\epsilon^d\mu^2(t))}\Xi(t)\Big)
\end{align*}

\paragraph{Scaling regimes.}

Our results for the collision numbers $C\super{n}(t)$ and $\bar C\super{n}(t)$ are related to our result for their jump rate $\Lambda\super{n}(t)$, and hence restricted to the regime $1/\epsilon^{d/2}\ll n \lle 1/\epsilon^d$. This restriction is not just technical, but a fundamental issue with our simple initial condition: independent Bernouilli distributions at each site. Indeed, in any scaling $1/\epsilon^{d/2}\gge n$, the nearest-neighbour measure $\Lambda\super{n}(t)$ can only converge to a random limit, and even if we could characterise that limit, it is still unclear whether we could deduce tightness of trajectories. Alternatively, one could change the initial condition to include a penalisation for too many occupied nearest neighbours. In that case, we expect it to be possible to study hydrodynamic limits up to scalings $n\sim \epsilon/\epsilon^{d/2}$, which is precisely the critical regime where the prefactors of the collision numbers
$C\super{n}(t), \bar C\super{n}(t)$ no longer vanish. In that case we expect Poisson point process limits. However, such alternative initial conditions are beyond the scope of the current paper.

\paragraph{Extensions.}
Our results and techniques provide a deeper understanding of the `simple' symmetric exclusion process, and they could potentially also be applied to more involved models. One obvious extension is to include weak asymmetries, i.e. non-trivial drifts due to an external force field. It would also be interesting to apply similar arguments to the inclusion process, introduced in \cite{giardina2007duality}, or the $k$-exclusion process introduced in \cite{schutz1994non}. Finally, since we explicitly measure the (inelastic) collisions, we can measure the lost kinetic energy and feed this back into the system to obtain kinetic-type exclusion processes~\cite{GutierrezHurtado2019}.

\appendix

\section{Functional analytic setting}
\label{sec:topo}

\subsection{Space of test functions} 
\label{subsec:test functions}

Equip the spaces of $n$-time continuous differentiable functions $C^n(\TT^d)$ with their usual uniform norms
\begin{equation*}
  \lVert\phi\rVert_{C^n(\TT^d)}:=\max_{\lvert q \rvert\leq n} \sup_{x\in\TT^d} \lvert D^q\phi(x)\rvert,
\end{equation*}where $D^q$ is the derivative of $\phi$ according to multi-index $q$. Of course we can omit the usual rapid decay condition because we work on the torus. Since the spaces $C^n(\TT^d)$ are nested, the projective limit is simply $\varprojlim_{n\geq0} C^n(\TT^d)=\bigcap_{n\geq0}C^n(\TT^d)=C^\infty(\TT^d)$, equipped with the projective limit topology, meaning a sequence or net of test functions $\phi\in C^\infty(\TT^d)$ converges whenever it converges in all $C^n$-norms.

This locally convex vector space forms a nuclear space, which follows from the isomorphism between periodic functions and rapidly decaying sequences of Fourier coefficients, see for example \cite[Th.~51.3]{Treves1967}. With this topology, $C^\infty(\TT^d)$ is automatically complete, separable, and Fr{\'e}chet, the latter meaning that it is metrisable by a (non-unique) translation-invariant metric, e.g. $\sum_{n=1}^\infty 2^{-n} \lVert \phi-\psi\rVert_{C^n(\TT^d)}/(1+\lVert \phi-\psi\rVert_{C^n(\TT^d)})$. 

Particularly helpful will be that the nuclearity and Fr{\'e}chet property imply that $C^\infty(\TT^d)$ is Montel~\cite[Sec.~III.7.2, Cor.2]{SchaefferWolff1999}: closed bounded subsets are automatically compact, and as a consequence the space is reflexive. Recall that a set $B$ is called bounded whenever $\sup_{\phi\in B}\lvert\langle\phi,\xi\rangle\rvert<\infty$ for all distributions $\xi\in C^\infty(\TT^d)^*$. Here and in this context $\langle\phi,\xi\rangle$ denotes the duality between test functions and distributions. 

\subsection{Distributions}
\label{subsec:distro}

The dual space of distributions $C^\infty(\TT^d)^*$ is naturally equipped with the \emph{strong topology}, generated by the seminorms $\lVert \xi\rVert_B:=\sup_{\phi\in B}\lvert\langle\phi,\xi\rangle\rvert$ indexed by bounded sets $B\subset C^\infty(\TT^d)$. Alternatively, $C^\infty(\TT^d)^*$ can be equipped with the \emph{weak-* topology}, defined by duality with $C^\infty(\TT^d)$. However, the Montel property of $C^\infty(\TT^d)^*$ is inherited from its predual \cite[Sec.~IV.5, Th.~9]{SchaefferWolff1999}, and as a consequence any weak-* converging sequence of distributions in $C^\infty(\TT^d)^*$ also converges strongly. Therefore, with regard to sequences, we do not need to distinguish between the two topologies.

An interesting fact that we will not use is that the test function space $C^\infty(\TT^d)$ is Fr{\'e}chet-Montel and hence ``distinguished''. Seeing $C^\infty(\TT^d)$ as a projective limit, this implies that its dual is the direct sum of the dual spaces: $C^\infty(\TT^d)^*=\varinjlim_{n\geq0} C^n(\TT^d)^*$, equipped with the corresponding ``total topology''.

The spaces of test functions $C^\infty(\lbrack0,T\rbrack\times\TT^d)$ and distributions $C^\infty(\lbrack0,T\rbrack\times\TT^d)^*$ are constructed in a similar fashion; to avoid confusion we denote their pairing by $\llangle\cdot,\cdot\rrangle$.

\subsection{Space-time white noise}

Observe that the ``center'' Hilbert space is simply given by $H^0(\lbrack0,T\rbrack\times\TT^d)=L^2(\lbrack0,T\rbrack\times\TT^d)$. Let $\Xi$ be the unique random variable in $C^\infty(\lbrack0,T\rbrack\times\TT^d)^*$ for which, using the Bochner-Minlos Theorem~\cite[Th.~2.1]{HidaSi2008}, \cite[Th.~7.13.9]{Bogachev2007}:
\begin{align*}
  \EE e^{i \llangle\phi,\Xi\rrangle} = e^{-\tfrac12\lVert\phi\rVert^2_{L^2(\lbrack0,T\rbrack\times\TT^d)}}.
\end{align*}
Then $\Xi_t:=\Xi\mathds1_{\lbrack0,t\rbrack}$ is called space-time white noise, where the product is to be understood in the sense that $\llangle\phi, \Xi_t\rrangle:=\llangle\phi \mathds1_{\lbrack0,t\rbrack},\Xi \rrangle$ for all test functions $\phi\in C^\infty(\lbrack0,T\rbrack\times\TT^d)$.

\subsection{Skorohod-Mitoma topology}
\label{subsec:Skorohod-Mitoma}

Let $D(0,T;C^\infty(\TT^d)^*)$ be the space of distribution-valued paths that are c{\`a}dl{\`a}g in $C^\infty(\TT^d)^*
$, equipped with the \emph{strong} topology (see above). We follow the classic construction from \cite{Mitoma1983} for the Skorohod topology for distribution-valued paths; this is not immediate since $C^\infty(\TT^d)^*$ is not metrisable. 

The strong topology comes with the family of seminorms $\lVert \xi\rVert_B$, and so we can at least define the family of Skorohod distances, indexed by the bounded sets $B\subset C^\infty(\TT^d) $,
\begin{multline*}
  d_B(\xi_1,\xi_2):=\\
  \inf_{\substack{\tau:\lbrack0,T\rbrack\to\lbrack0,T\rbrack\\\text{strictly increasing,} \\ \text{continuous, surjective} }}
  \bigg\{
    \sup_{t\in\lbrack0,T\rbrack} \lVert \xi_1(t)-\xi_2(\tau(t))\rVert_B 
    + 
    \sup_{s,t\in\lbrack0,T\rbrack:s\neq t} \Big\lvert \log\frac{\tau(t)-\tau(s)}{t-s}\Big\rvert 
  \bigg\}.
\end{multline*}
The Skorohod-Mitoma topology on $D(0,T;C^\infty(\TT^d)^*)$ is generated by the family of the metrics $(d_B)_B$, that is: a sequence converges $\xi^n\to\xi$ in $D(0,T;C^\infty(\TT^d)^*)$ whenever $d_B(\xi^n,\xi)\to0$ for all bounded sets $B\subset C^\infty(\TT^d)$. 
The Montel property of distributions carries over to tightness, see Mitoma's Theorem~\ref{th:mitoma}.

\section{Probabilistic setting}
\subsection{On Some Martingales for Markov Processes}
Let \((Z(t))_{t\in[0,T]}\) be a distribution-valued Markov process with c{\`a}dl{\`a}g paths in \(D(0,T;C^\infty(\TT^d)^*)\), and let its generator be denoted by \(\Q\), acting on a subset \(\Dom \Q^n\) of the Borel-measurable functions. Throughout, we tacitly assume that the linear functionals \(\langle\phi,\cdot\rangle: z \mapsto\langle\phi,z\rangle\) belong to \(\Dom\Q\) for all \(\phi\in C^\infty(\TT^d)\) (or for some suitable subset thereof).

It is well known (see, e.g., \cite{KipnisLandim1999}) that for any \(\phi \in C^\infty(\TT^d)\), the process
\begin{equation}
  M^{\phi}(t) := \langle \phi,Z(t)\rangle - \langle \phi,Z(0)\rangle - \int_0^t (\Q \langle\phi,\cdot\rangle)(Z(s))\,ds
\label{eq:Dynkin_martingale}
\end{equation}
is an \(\mathcal{F}_t\)-martingale. Its quadratic variation is characterized by the fact that the process
\begin{equation}
  N^{\phi}(t) := \left(M^{\phi}(t)\right)^2 - \int_0^t (\Gamma_2 \langle \phi,\cdot\rangle)(Z(s))\,ds
\label{eq:Dynkin_QuadVar}
\end{equation}
is itself an \(\mathcal{F}_t\)-martingale.

Here, the operator \(\Gamma_2\) (the carré du champ or ``square field'' operator) associated to \(\Q\) is defined, for all \(F \in \Dom\Q\), by
\begin{equation}
    \Gamma_2(F(Z)) := \Gamma(F(Z)) := \Q(F^2(Z)) - 2F(Z)\, \Q(F(Z)).
\label{eq:square_field}
\end{equation}

A less commonly discussed fact is that the quadratic variation of the process \(N^{\phi}(t)\) can also be explicitly characterized. To this end, for \(k \in \{1,2,3,4\}\), define the operators
\begin{equation}
    \Gamma_k(F(Z)) := \sum_{j=0}^k \binom{k}{j}(-1)^j F^j(Z)\, \Q(F^{k-j}(Z)).
\label{eq:k_field}
\end{equation}

\begin{remark}\label{rem.k_field}
Notice that if the generator $\Q$ has the form
\begin{equation*}
    \Q F(Z) = \sum_{Z^\prime} c(Z,Z^\prime) (F(Z^\prime) -F(Z)),
\end{equation*}
then the $k$-field operator $\Gamma_k$ can be written as
\begin{equation}\label{eq. k-field power}
   \Gamma_k F(Z) = \sum_{Z^\prime} c(Z,Z^\prime) (F(Z^\prime) -F(Z))^k. 
\end{equation}
\end{remark}

\begin{lemma}\label{lem: quadvar of quadvar}
    The process \(N^{\phi}(t)\) is a martingale whose quadratic variation is given by
\begin{align*}
 &\int_0^t (\Gamma_4 \langle \phi,\cdot\rangle)(Z(s)) \,ds 
 + 2\left(  \int_0^t (\Gamma_2 \langle \phi,\cdot\rangle)(Z(s)) \,ds \right)^2 \\
 &\quad + 4 M^{\phi}(t) \int_0^t (\Gamma_3 \langle \phi,\cdot\rangle)(Z(s)) \,ds 
 + 4 \left(M^{\phi}(t)\right)^2 \int_0^t (\Gamma_2 \langle \phi,\cdot\rangle)(Z(s)) \,ds.
\end{align*}
\end{lemma}

For further details and a complete proof, see \cite{jara2023stochastic}. We emphasize that this characterization holds not only for distribution-valued Markov processes but also in greater generality.

\printbibliography

@article{BDSGJLL2015MFT,
	author = {Bertini, L. and De Sole, A. and Gabrielli, D. and Jona--Lasinio, G. and Landim, C.},
	journal = {Reviews of Modern Physics},
	number = {2},
	title = {Macroscopic Fluctuation Theory},
	volume = {87},
	year = {2015}}

@article{BDSGJLL2006,
  title={Non equilibrium current fluctuations in stochastic lattice gases},
  author={Bertini, L. and De Sole, A. and Gabrielli, D. and Jona-Lasinio, G. and Landim, C.},
  journal={Journal of statistical physics},
  volume={123},
  pages={237--276},
  year={2006},
  publisher={Springer}}

@article{BDSGJLL2007,
  title={Large deviations of the empirical current in interacting particle systems},
  author={Bertini, L. and De Sole, A. and Gabrielli, D. and Jona-Lasinio, G. and Landim, C.},
  journal={Theory of Probability \& Its Applications},
  volume={51},
  number={1},
  pages={2--27},
  year={2007},
  publisher={SIAM}}

@book{Billingsley1999,
  author = {Billingsley, P.},
  title = {Convergence of probability measures},
  publisher = {Wiley},
  address = {New York, NY,USA},
  edition = {2nd},
  year = {1999}}

@book{Bogachev2007,
	Address = {Berlin, Germany},
	Author = {Bogachev, V.I.},
	Publisher = {Springer},
	Title = {Measure theory {V}ol.~{I} and~{II}},
	Year = {2007}}

@misc{DalangSanzSole2024TR,
  title={Stochastic Partial Differential Equations, Space-time White Noise and Random Fields},
  author={Dalang, R.C. and Sanz-Sol{\'e}, M.},
	howpublished = {\href{https://arxiv.org/abs/2402.02119}{ArXiv Preprint~2402.02119}},
  year={2024}}

@article{Dean1996,
  title={Langevin equation for the density of a system of interacting {L}angevin processes},
  author={Dean, D.S.},
  journal={Journal of Physics A: Mathematical and General},
  volume={29},
  number={24},
  pages={L613},
  year={1996}}

@book{FengKurtz06,
	author = {Feng, J. and Kurtz, T. G.},
	publisher = {American Mathematical Society},
	series = {Mathematical Surveys and Monographs},
	title = {{Large deviations for stochastic processes}},
	volume = {131},
	year = {2006}}

@inproceedings{ferrari1988non,
  title={Non equilibrium fluctuations for a zero range process},
  author={Ferrari, P.A. and Presutti, E. and Vares, M.E.},
  booktitle={Annales de l'{IHP} {P}robabilit{\'e}s et statistiques},
  volume={24},
  number={2},
  pages={237--268},
  year={1988}}

@article{GutierrezHurtado2019,
  author = {Guti{\'e}rrez-Ariza, C. and Hurtado, P.I.},
  title = {The kinetic exclusion process: a tale of two fields},
  journal = {Journal of Statistical Mechanics: Theory and Experiment},
  volume = {2019},
  number = {10},
  pages = {103203},
  year = {2019}}

@book{HidaSi2008,
  author={Hida, T. and Si, S.},
  title={Lectures on white noise functionals},
  publisher = {World Scientific},
  address ={Singapore},
  year = {2008}}

@article{Jordan1998,
	author = {Jordan, R. and Kinderlehrer, D. and Otto, F.},
	journal = {SIAM journal on mathematical analysis},
	number = {1},
	pages = {1--17},
	title = {The variational formulation of the {F}okker--{P}lanck equation},
	volume = {29},
	year = {1998}}

@book{KipnisLandim1999,
	address = {Berlin-Heidelberg, Germany},
	author = {Kipnis, C. and Landim, C.},
	publisher = {Springer},
	title = {Scaling limits of interacting particle systems},
	year = {1999}}

@article{KipnisOllaVaradhan1989,
  title={Hydrodynamics and large deviation for simple exclusion processes},
  author={Kipnis, C. and Olla, S. and Varadhan, S.R.S.},
  journal={Communications on Pure and Applied Mathematics},
  volume={42},
  number={2},
  pages={115--137},
  year={1989},
  publisher={Wiley Online Library}}

@book{Liggett1985,
  author={Liggett, T.M.},
  title={Interacting particle systems},
  year={1985},
  publisher={Springer},
  address={Berlin-Heidelberg, Germany}}

@article{Mitoma1983,
  title={Tightness of probabilities on {$C([0, 1];\mathscr{S}')$} and {$D([0, 1];\mathscr{S}')$}},
  author={Mitoma, I.},
  journal={The Annals of Probability},
  pages={989--999},
  year={1983}}

@article{PfaffelhuberPopovic2015,
  title={Scaling limits of spatial compartment models for chemical reaction networks},
  author={Pfaffelhuber, P. and Popovic, L.},
  journal = {The Annals of Applied Probability},
  volume = {25},
  number = {6},
  pages = {3162 -- 3208},
  year = {2015}}

@article{PulvirentiSimonella2017,
  author = {Pulvirenti, M. and Simonella, S.},
  title = {The {B}oltzmann–Grad limit of a hard sphere system: analysis of the correlation error},
  journal = {Inventiones mathematicae},
  volume = {207},
  pages = {1135--1237},
  year = {2017}}

@book{RevuzYor1999,
  author={Revuz, D. and Yor, M.},
  title = {Continuous martingales and {B}rownian motion},
  edition = {3rd},
  publisher = {Springer},
  address = {Berlin, Germany},
  year = {1999}}

@book{Treves1967,
  author={Treves, F.},
  title = {Topological vector spaces, distributions and kernels},
  publisher = {Academic Press},
  address = {New York, NY, USA},
  year = {1967}}

@book{demasi2006mathematical,
  title={Mathematical methods for hydrodynamic limits},
  author={DeMasi, Anna and Presutti, Errico},
  year={2006},
  publisher={Springer}
}

@article{spitzer1970interaction,
  title={Interaction of Markov processes},
  author={Spitzer, Frank},
  journal={Advances in Mathematics},
  volume={5},
  number={2},
  pages={246--290},
  year={1970},
  publisher={Elsevier}
}

@Article{KOLESNIKOV2006382,
  author    = {Alexander V. Kolesnikov},
  title     = {Mosco convergence of Dirichlet forms in infinite dimensions with changing reference measures},
  journal   = {Journal of Functional Analysis},
  year      = {2006},
  volume    = {230},
  number    = {2},
  pages     = {382-418},
  month     = jan,
  publisher = {Elsevier BV},
}

@book{Rebolledo1979,
  title={La m{\'e}thode des martingales appliqu{\'e}e {\`a} l'{\'e}tude de la convergence en loi de processus},
  author={Rebolledo, R.},
  year={1979},
  publisher={Soci{\'e}t{\'e} math{\'e}matique de France},
}

@book{SchaefferWolff1999,
  author={Schaefer, H.H. and Wolff, M.P.},
  title={Topological Vector Spaces},
  publisher={Springer},
  address={New York, N.Y., U.S.A.},
  edition={2nd},
  year={1999}}

@article{yau1991relative,
  title={Relative entropy and hydrodynamics of Ginzburg-Landau models},
  author={Yau, Horng-Tzer},
  journal={Letters in Mathematical Physics},
  volume={22},
  number={1},
  pages={63--80},
  year={1991},
  publisher={Springer}
}

@book{ethier_kurtz_1986,
    title={Markov Processes: Characterization and Convergence},
    author={Ethier, S.N. and Kurtz, T.G.},
    year={1986},
    publisher={John Wiley \& Sons, Inc.},
    series={Wiley Series in Probability and Statistics}}

@article{jara2023stochastic,
  title={The stochastic heat equation as the limit of a stirring dynamics perturbed by a voter model},
  author={Jara, M. and Landim, C.},
  journal={The Annals of Applied Probability},
  volume={33},
  number={6A},
  pages={4163--4209},
  year={2023},
  publisher={Institute of Mathematical Statistics}
}

@article{ravishankar1992fluctuations,
  title={Fluctuations from the hydrodynamical limit for the symmetric simple exclusion in $\mathbb{Z}^d$},
  author={Ravishankar, K.},
  journal={Stochastic Process. Appl},
  volume={42},
  number={1},
  pages={31--37},
  year={1992}
}

@article{giardina2007duality,
  title={Duality and exact correlations for a model of heat conduction},
  author={Giardina, C. and Kurchan, J. and Redig, F.},
  journal={Journal of mathematical physics},
  volume={48},
  number={3},
  year={2007},
  publisher={AIP Publishing}
}

@article{schutz1994non,
  title={Non-Abelian symmetries of stochastic processes: Derivation of correlation functions for random-vertex models and disordered-interacting-particle systems},
  author={Sch{\"u}tz, Gunter and Sandow, Sven},
  journal={Physical Review E},
  volume={49},
  number={4},
  pages={2726},
  year={1994},
  publisher={APS}
}
\end{document}